%% file: HigherRankCatalanPolynomialsARXIV.tex
\documentclass[11pt,reqno]{amsart}

\input{packages-Catalan.tex}

\input{macros-Catalan.tex}

\theoremstyle{plain}
\newtheorem{theorem}{Theorem}[section]
\newtheorem{corollary}[theorem]{Corollary}
\newtheorem{proposition}[theorem]{Proposition}
\newtheorem{lemma}[theorem]{Lemma}

\theoremstyle{definition}
\newtheorem{definition}[theorem]{Definition}
\newtheorem{conjecture}[theorem]{Conjecture}
\theoremstyle{remark}
\newtheorem{remark}[theorem]{Remark}

\begin{document}


\title[Higher Catalan]{Higher rank $(q,t)$-Catalan Polynomials, Affine Springer Fibers, and a Finite Rational Shuffle Theorem}

\author[Gonz\'{a}lez]{Nicolle Gonz\'{a}lez}
\address{Department of Mathematics,  University of California Berkeley , CA 94720-3840, U.S.A.}
\email{nicolle@math.berkeley.edu}
\author[Simental]{Jos\'e Simental}
\address{Instituto de Matem\'aticas, Universidad Nacional Aut\'onoma de M\'exico. Ciudad Universitaria, CDMX,  M\'exico}
\email{simental@im.unam.mx}
\author[Vazirani]{Monica Vazirani}
\address{Department of Mathematics,  University of California Davis, 95616, U.S.A}
\email{vazirani@math.ucdavis.edu }

\thanks{}


\date{\today}



\begin{abstract}
We introduce the higher rank $(q,t)$-Catalan polynomials and prove they equal truncations of the Hikita polynomial to a finite number of variables. Using affine compositions and a certain standardization map,  we define a $\mathtt{dinv}$ statistic on rank $r$ semistandard $(m,n)$-parking functions and prove $\mathtt{codinv}$ counts the dimension of an affine space in an affine paving of a parabolic affine Springer fiber.  Combining these results, we give a finite analogue of the Rational Shuffle Theorem in the context of double affine Hecke algebras.  Lastly, we also give a Bizley-type formula for the higher rank Catalan numbers in the non-coprime case.

\end{abstract}

\maketitle

\section{Introduction}

Some of the most celebrated results in recent years in algebraic combinatorics have been the so called \emph{shuffle theorems}. Originally arising in Garsia and Haiman's search for the Frobenius character of the space of diagonal harmonics, these theorems have seen various generalizations with deep and unexpected connections to Macdonald polynomials, the homology of torus knots, the geometry of Gieseker varieties, the K-theory of Hilbert schemes and other areas in representation theory, $(q,t)$-Catalan combinatorics, and algebraic geometry. At the core of the theorems is the search for an expression of a particular action of the \emph{elliptic Hall algebra} on the ring of symmetric functions in an infinite number of variables in terms of important combinatorial objects known as \emph{parking functions}. Deeply intertwined in this story are the \emph{rational $(q,t)$-Catalan numbers}, which arise naturally in all the algebraic, geometric, topological, and combinatorial settings in which the various connections lie. In this paper we generalize several of these constructions as well as deepen some of the geometric and representation theoretic connections. 

We begin by introducing the \emph{rank r rational $(q,t)$-Catalan polynomials}, a family of symmetric polynomials in $r$-variables that specialize to the rational $(q,t)$-Catalan numbers. To do this, we propose new combinatorial objects called \emph{affine compositions} which we show are in bijection with the set of \emph{semistandard parking functions} introduced previously in \cite{EKLS}. This bijection generalizes a previous construction of Gorsky-Mazin-Vazirani \cite{GMV} between certain
affine permutations and $(m,n)$-parking functions,
where $m$ and $n$ are coprime positive integers. Utilizing this bijection, we construct a \emph{standardization} procedure for semistandard parking functions which allows us to define traditional combinatorial statistics like $\dinv$ and $\area$. A common theme in algebraic combinatorics is finding expressions for important functions in terms of combinatorial statistics. A family of functions admitting such a combinatorial description are the \emph{Hikita polynomials}, series in infinite variables which appear most notably in the statement of the \emph{Rational Shuffle Theorem}. One of our main theorems is that the rank $r$ rational $(q,t)$-Catalan polynomials correspond precisely to a finite variable truncation of the Hikita polynomials. 

Connecting our work to the world of \emph{affine Springer fibers}, we then describe how the $\dinv$ statistic actually arises naturally in a geometric setting. Namely, in \cite{LusztigSmelt} Lusztig and Smelt prove that certain Springer fibers $\mathcal{F}_{m,n}$ in the affine flag variety can be paved by $m^{n-1}$ affine cells. Hikita \cite{hikita} expanded this further by providing a combinatorial bijection between the cells in $\mathcal{F}_{m,n}$ and $(m,n)$-parking functions, as well as computing the dimensions of the cells in terms of combinatorial statistics on parking functions.
Gorsky-Mazin-Vazirani \cite{GMV} then combined these ideas and proved that the affine Springer fiber $\mathcal{F}_{m,n}$ admits a paving by affine cells indexed by so called \emph{m-stable} affine permutations. We generalize these constructions further and consider a larger variety $X$ consisting of a disjoint union of certain affine Springer fibers. Another of our main results is that the variety $X$ admits an affine paving by affine cells indexed by the set of semistandard parking functions, where the dimension of each affine cells is  precisely equal to the $\codinv$ of the corresponding semistandard parking function.

Finally, we turn towards representation theory. The classical and rational shuffle theorems, conjectured by Haglund et. al. \cite{HHLRU} and Gorsky-Negu\c{t} \cite{GorskyNegut} and proven by Carlsson-Mellit \cite{CM} and Mellit \cite{Mellit-rational}, respectively, express a particular geometric action of the elliptic Hall algebra $\EHA$ on the element $1$ in the ring of symmetric functions as a combinatorial sum over $(m,n)$-parking functions. In the classical case, this sum agrees with the action of the nabla operator $\nabla$ on the $n$th elementary symmetric function. In the rational setting, this action yields the Hikita polynomial. Relying on Schiffmann and Vasserot's beautiful work describing the elliptic Hall algebra as an inverse limit of \emph{spherical double affine Hecke algebras} $\DAHA{r}$, we transport the geometric action of $\EHA$ to the realm of symmetric polynomials in a \emph{finite} number of variables $\Symr$ to construct a new representation of the spherical double affine Hecke algebra on $\Symr$. We then show this geometric representation of $\DAHA{r}$ is isomorphic to its usual polynomial representation via a `truncated' plethysm that intertwines these actions. Our final main result is that this new $\DAHA{r}$ action yields the rank $r$ rational $(q,t)$-Catalan polynomials and thus recovers, under the inverse limit, the action of $\EHA$ that produces the Hikita polynomial. That is, we prove a \emph{finite rational shuffle theorem}.  

In all of the constructions above, the integers $(m,n)$ are always considered to be coprime and positive. However, many of the constructions above have been generalized to the non-coprime or even non-integer case. For instance, in recent years a shuffle theorem for any slope was proven by Blasiak et. al. \cite{BHMPS}. In this article we also begin to study this direction by giving a Bizley-type formula for the \emph{higher rank rational Catalan numbers} when $m,n$ are not coprime.  We hope this may help pave the way for an eventual DAHA shuffle theorem for all slopes.  

\subsection{Semistandard Parking Functions and Higher Rank Catalan Numbers} Dyck paths are one of the most studied objects in combinatorics. Although they initially arose in the study of ballot problems, they have a surprising number of connections to representation theory, low-dimensional topology and algebraic combinatorics. Given positive integers $m$ and $n$, an \newword{(m,n)-Dyck path} is a lattice path consisting of south and east steps from $(0,m)$ to $(n,0)$ that stays weakly below the diagonal line $mx + ny = mn $. Denote by $\dyck{m,n}$ the set of all $(m,n)$-Dyck paths. The \emph{rational Catalan number} $C_{(m,n)}$ is defined as the cardinality of this set. When $m = n$ or $m = n+1$, one recovers the usual Catalan numbers $C_n = \frac{1}{n+1}{{2n}\choose{n}}$. When $m = kn+1$, one obtains the so-called Fuss-Catalan numbers. 

Some of the connections between Dyck paths and various fields are better appreciated when one looks at certain related objects that arise as labelings of these paths known as \emph{parking functions}. These functions, introduced originally by Konheim and Weiss \cite{parking} in relation to hashing problems and generalized to the rational-slope case by Armstrong, Loehr and Warrington \cite{ALW}, are given by certain bijective fillings from the set $\lbrace 1, \dots m \rbrace$ to the vertical steps of the Dyck path. In this paper, we study combinatorial objects introduced by Simental alongside Etingof-Krylov-Losev \cite{EKLS} that interpolate between Dyck paths and parking functions called \emph{semistandard parking functions}, defined as follows.  

Given $\D \in \dyck{m,n}$, denote by $\ver(\D)$ the set of vertical (i.e. south) steps of $\D$, so that $|\ver(\D)| = m$. 
For any positive integers, $r,m$ and $n$, a \newword{rank $r$ semistandard parking function} is a pair $(\D, \Upsilon)$, where 
\begin{enumerate}
\item $\D \in \dyck{m,n}$ is an $(m,n)$-Dyck path,  and
\item $\Upsilon$ is a function $\Upsilon: \ver(\D) \to \{1, \dots, r\}$,
\end{enumerate}
such that if $v_i, v_{i+1}$ are two \emph{consecutive} vertical steps with $v_i$ directly above $v_{i+1}$, then $\Upsilon(v_i) \leq \Upsilon(v_{i+1})$ (see Definition \ref{def:sspf}).
We denote by $\sspf{r}{m,n}$ the set of rank $r$ semistandard parking functions and define the \newword{rank $r$ rational Catalan number} $C^{(r)}_{(m,n)}$ as the cardinality of this set. While the following theorem was originally proven in \cite{EKLS} using representation-theoretic techniques, in Theorem \ref{thm:coprime sspf} we give a new elementary proof. 
\begin{theorem} [Thm. \ref{thm:coprime sspf}] \label{thm:coprimeIntro}
For any positive integers $m,n,r$ with $m,n$ coprime, 
\[
C^{(r)}_{(m,n)} = \frac{1}{n}\binom{nr+m-1}{m}.
\]
\end{theorem}

As mentioned above, the semistandard parking functions interpolate between Dyck paths and $(m,n)$-parking functions in the following sense. If $r = 1$, the function $\Upsilon$ in Definition \ref{def:sspf} is forced to be the constant function $1$, so we have an obvious bijection $\sspf{1}{m,n} \to \dyck{m,n}$, with $C^{(1)}_{(m,n)} = C_{(m,n)}$. Conversely, if $r=m$ and $\Upsilon$ is a bijection then the resulting object is an \newword{$(m,n)$-parking function}. In the case when $m$ and $n$ are coprime, these objects were studied in \cite{EKLS} in connection to the representation theory of quantized Gieseker varieties, of which the Hilbert scheme of points on $\C^2$ is the ``rank $1$" case. Connections between Catalan combinatorics and the representation theory of quantized Hilbert schemes (i.e. type $A$ rational Cherednik algebras) have been extensively studied in the literature, see e.g. \cite{BEG, gordon, GordonGriffeth, GMV, GORS}.

\subsection{Higher Rank $(q,t)$-Catalan Polynomials} \label{sec:introCatalanPol}
The first goal of this paper is to produce a multiparametric deformation $C^{(r)}_{(m,n)}(x_1, \dots, x_r; q,t)$ of the rank $r$ rational Catalan number $C^{(r)}_{(m,n)}$ in the spirit of the usual $(q,t)$-Catalan numbers $C_{n}(q,t)$ of \cite{ALW, GHCatalan, Loehr}. There are several reasons to expect a combinatorially meaningful multiparametric deformation of $C^{(r)}_{(m,n)}$ to exist. First of all, the number $C^{(r)}_{(m,n)}$ arises in \cite{EKLS} as the dimension of a graded $\GLr$-representation that is conjectured to admit a compatible filtration: $C^{(r)}_{(m,n)}(x_1, \dots, x_r;q,t)$ is then, conjecturally, its bigraded $\GLr$-character. Moreover, it was also shown in \cite[Prop. 2.33]{EKLS} that for every divisor $d$ of $r$ the following expression of $q$-numbers, 
\[
\frac{[d]_q}{[nd]_q}\left[\begin{matrix}nr + m - 1 \\ m \end{matrix}\right]_q
\]
is a polynomial in $q$ with nonnegative integer coefficients that recovers the rank $r$ rational Catalan numbers at $q=1$. Thus, $C^{(r)}_{(m,n)}$ admits many $1$-parametric deformations, each of which should arise as a different specialization of a single multiparametric deformation.

We proceed to define the rank $r$ Catalan polynomial by introducing the usual statistics on the set $\sspf{r}{m,n}$ of $\area$, $\dinv$ and weight. The $\area$ statistic corresponds to the variable $q$, the $\dinv$ statistic to the variable $t$, and the weight statistic is $\Z^r_{\geq 0}$-valued and accounts for the variables $x_1, \dots, x_r$.  The $\area$ and weight statistics are defined as for parking functions and can be easily obtained from $(\D, \Upsilon)$. The $\dinv$ statistic, on the other hand, requires more work.

This statistic has been defined and proven to be meaningful in the special cases of Dyck paths and $(m,n)$-parking functions mentioned above. We denote the set of $(m,n)$-parking functions by $\park{m,n}$ and leverage the knowledge of a meaningful dinv statistic on the set $\park{m,n}$ to define this statistic on $\sspf{r}{m,n}$. To do so, we generalize several notions that were introduced by Gorsky-Mazin-Vazirani in \cite{GMV}. In particular, we define certain functions which we call \newword{affine compositions}. These functions naturally generalize \emph{affine permutations} and, under some stabilization conditions, are in bijective correspondence with the
finite
 set $\sspf{r}{m,n}$. That is, if $\affc{}^n(m, r)$ denotes the set of \newword{n-stable affine compositions} (see Definition \ref{def:affcomp}), given any $(\D, \Upsilon) \in \sspf{r}{m,n}$ we construct an explicit function $f_{(\D, \Upsilon)} \in \affc{}^n(m, r)$ and prove the following. 
\begin{theorem}[Thm. \ref{thm:bijection}] \label{thm:bijectionIntro}
The map $\cA: \sspf{r}{m,n} \to \affc{}^n(m, r)$ sending  $(\D, \Upsilon) \mapsto f_{(\D, \Upsilon)}$ is a weight preserving bijection.
\end{theorem}
Combining this map with a bijection between affine compositions of a fixed weight $\bw$ and minimal coset representatives in $\minlength{\bw}{m}{}$, we construct a \newword{standardization map}:
\[
\std: \sspf{r}{m,n} \to \park{m,n}
\]
and for any $(\D, \Upsilon) \in \sspf{r}{m,n}$ define $\dinv(\D, \Upsilon) := \dinv(\std(\D, \Upsilon))$ (for a classical description in terms of laser rays of slope $\frac{-m}{n}$ see Appendix \ref{sec:appendix}). With this in hand, for $m,n$ any coprime positive integers, the \newword{rank r $(q,t)$-Catalan polynomial} is defined as:
\[
C^{(r)}_{(m,n)}(x_1, \dots, x_r;q,t) := \sum_{(\D, \Upsilon) \in \sspf{r}{m,n}}q^{\area(\D,\Upsilon)}t^{\dinv(\D, \Upsilon)}x^{\we(\D, \Upsilon)}. 
\]
These polynomials generalize the usual \emph{rational $(q,t)$-Catalan numbers} $C_{(m,n)}(q,t)$ in the sense that when $r=1$, we have $C^{(r)}_{(m,n)}(x_1;q,t) =
C_{(m,n)}(q,t) x_1^m$.

\subsection{The Hikita polynomial}\label{subsec:Hikita}
The ring of \emph{diagonal coinvariants} $DR_n$, which is isomorphic to the space of \emph{diagonal harmonics}, has been a central figure in some of the most exciting results in algebraic combinatorics and representation theory in the last two decades. Its bigraded Frobenius character was proven by Haiman \cite{Haiman} to be $\nabla e_n$, where $\nabla$ denotes the \emph{nabla operator} ubiquitous in theory of Macdonald polynomials and $(q,t)$ combinatorics. The classical \emph{Shuffle Theorem}, conjectured by Haglund
et. al.
\cite{HHLRU} and proven fifteen years later by Carlsson-Mellit \cite{CM}, gives a combinatorial formula for this character as a sum over parking functions $\park{n}$. 

This combinatorial formula was interpreted geometrically and generalized by Hikita by relating it to affine Springer fibers associated to regular semisimple nil elliptic elements. In particular, Hikita showed that the Frobenius character of the Borel-Moore homology of certain affine Springer fibers is given by an analogous combinatorial formula now indexed by \emph{$(m,n)$}-parking functions and expressed in terms of Gessel's quasisymmetric functions $Q_S(X)$. This resulting symmetric function is eponymously called the \newword{Hikita polynomial} and is defined as:
\[ \hik_{(m,n)}( X; q,t) := \sum_{(\D, \varphi) \in \park{m,n}}q^{\area(\D, \varphi)}t^{\dinv(\D, \varphi)}Q_{\des(\omega^{-1})}(X),\]
where $\omega$ is an affine permutation corresponding to $(\D, \varphi)$ under the map $\cA$, and $\des(\omega^{-1})$ is the \emph{descent set} of the permutation $\omega^{-1}$, see Section \ref{sec:parking} for a precise definition. Thus, the Hikita polynomial can equivalently be defined as the Frobenius character of the $\C [S_{m}]$ module $\C\!\park{m,n}$, where $m,n$ are coprime.

While it is well known that the coefficient of $Q_{\emptyset}$ in  $\hik_{(m,n)}( X; q,t)$ is given by $C_{m,n}(q,t)$, a more explicit relation between the Hikita polynomials and the $(q,t)$-Catalan numbers is only known in some special cases \cite{KK-Hikita}. Our first main theorem is that our higher rank Catalan polynomials are precisely truncations of the Hikita polynomial to a finite number of variables. Hence, for any coprime $m,n$, $C^{(r)}_{(m,n)}(x_1, \dots, x_r; q,t)$ is a symmetric polynomial that is both Schur positive and $q,t$-symmetric. 
\begin{theorem}[Thm. \ref{thm:Cat=Hik}]\label{thm:main1}
For any $r \in \Z_{>0}$ and coprime $m,n \in \Z_{>0}$, we have that 
\[C^{(r)}_{(m,n)}(x_1, \dots, x_r; q,t) = \hik_{(m,n)}(X;q,t)\vert_{x_1, \dots, x_r}.\]
\end{theorem}

It turns $DR_n$ can also be recovered from the representation theory of the rational Cherednik algebra $H_{n,c}$ of type A. In particular, Berest-Etingof-Ginzburg \cite{BEG} proved that when $c = \pm m/n$ with $m,n$ coprime, $H_{n,c}$ has a unique finite dimensional representation $L_{m/n}$. It is a theorem of Gordon \cite{gordon} that $ DR_n \cong L_{\frac{n+1}{n}} \otimes \sgn$. This framework was used by Gorsky-Negu\c{t} (see \S \ref{subsec:finiteshuffle}) to give a rational generalization of the Shuffle Theorem and conjecture that the bigraded Frobenius character of  $L_{m/n} \otimes \sgn$ is the Hikita polynomial \cite{GorskyNegut}. 

As mentioned in \S \ref{sec:introCatalanPol}, it is conjectured that a multiparametric deformation of the higher rank Catalan numbers $C_{(m,n)}^{(r)}$ corresponds to certain $\GLr$ characters. 
Indeed, in \cite{EKLS}, they considered a family of representations $L_{m/n}^{(r)}$ of quantized Gieseker varieties which as $\mathbb{C}^{\times} \times \GLr$-modules
are isomorphic to the $S_m$-invariants of $L_{n/m}\otimes (\mathbb{C}^{r*})^{\otimes m}$. It is expected that $L_{m/n}^{(r)}$ admits a compatible filtration, and thus $\gr L_{m/n}^{(r)}$ is a bigraded $\GLr$-representation. 

\begin{conjecture}
For any positive integers $m,n,r$ with $\gcd(m,n)=1$, the bigraded $\GLr$-character of $\gr L_{m/n}^{(r)}$ is given by $C^{(r)}_{(m,n)}(x_1, \dots, x_r; q,t)$.
\end{conjecture}

\subsection{Affine Springer fibers}\label{sec:asf intro}The $\dinv$ statistic on $\sspf{r}{m,n}$ can also be obtained via a geometric construction using affine Springer fibers. Let us fix coprime positive integers $m, n$.  We will consider a nill-elliptic element (in the sense of \cite{KazhdanLusztig}) $\gamma = \gamma_{(m,n)}$ of the loop Lie algebra $\mathfrak{sl}_{m}[[\epsilon]]$ and various affine Springer fibers for it. More precisely, consider a composition $\mathbf{w} := (w_1, \dots, w_r)$ of $m$ with exactly $r$ parts, some of which may be zero \footnote{compositions whose parts are allowed to be 0 are often called \emph{weak compositions}
\label{footnote-composition}.
Since all compositions in this paper are weak, we will omit the adjective weak and simply call them compositions.}.
Associated to this composition, we have a parahoric subgroup $P^{\bw}$ of $\mathrm{SL}_{m}[[\epsilon]]$ and the partial affine flag variety:
\[
\Fl^{\bw} := \mathrm{SL}_{m}((\epsilon))/P^{\bw}.
\]
The variety $\mathcal{F}\ell^{\bw}$ is an ind-projective variety, that is, an inverse limit of projective varieties, and it is infinite-dimensional. However,  due to results of Kazhdan-Lusztig and Lusztig-Smelt, \cite{KazhdanLusztig, LusztigSmelt}, the affine Springer fiber
\[
X^{\bw}_{(m,n)} :=  \{[g] \in \Fl^{\bw} \mid g^{-1}\widetilde{\gamma}g \in \mathrm{Lie}(P^{\bw})\}
\]
is a projective variety of finite dimension $(m-1)(n-1)/2$. We will consider the space:
\[
X := \bigsqcup_{\mathbf{w}}X^{\bw}_{(m,n)}
\]
where the disjoint union runs over all compositions of $m$ with exactly $r$ parts.
This $X$ can be constructed as a single \emph{generalized} affine Springer fiber, similar to \cite[Section 8.4]{GSV}.
\begin{theorem}[Thm. \ref{thm:geo dinv}] \label{thm:main2} 
The variety $X$ admits an affine paving with affine cells indexed by the set $\sspf{r}{m,n}$. Moreover, if $C^{(\D, \Upsilon)}$ is the affine cell indexed by the semistandard parking function $(\D, \Upsilon)$ then
\[
\dinv(\D, \Upsilon) = \frac{(m-1)(n-1)}{2} - \dim_\C(C^{(\D, \Upsilon)}).
\]
\end{theorem}

In order to prove Theorem \ref{thm:main2}, we make use of the bijection $\cA: \sspf{r}{m,n} \to \affc{n}(m,r)$ described above since the affine cells in Theorem \ref{thm:main2} are more naturally indexed by affine compositions. When restricted to parking functions, the map $\cA$ becomes the bijection studied in detail by Gorsky, Mazin and the third author
in \cite{GMV}, after work of Hikita in \cite{hikita}. Thus, our construction may be seen as a parabolic version of theirs.

\subsection{DAHA and a Finite Rational Shuffle Theorem}\label{subsec:finiteshuffle} 
Although the Shuffle Theorem is an inherently combinatorial statement, its rational generalization originally arose from topological considerations. 
In their search for a description of the triply graded Khovanov–Rozansky homology of the $(m, n)$- torus knot, Gorsky-Oblomkov-Rasmussen-Shende \cite{GORS} conjectured that this homology could be extracted from the unique finite dimensional representation $L_{m/n}$ of the rational Cherednik algebra $H_{n,c}$ previously mentioned in \S \ref{subsec:Hikita}. This connection led Gorsky and Negu\c{t} \cite{GorskyNegut} to make the remarkable conjecture that the Hikita polynomial $\hik_{(m,n)}(X; q,t)$ could be obtained from the \emph{geometric} action of Schiffmann-Vasserot \cite{SchiffmannVasserotK} of the \newword{elliptic Hall algebra} $\EHA$ on symmetric functions $\Sym$. Proved by Mellit \cite{Mellit-rational}, the \newword{Rational Shuffle Theorem} states that for each generator $P_{(m,n)} \in \EHA$ with $m,n$ coprime:
\begin{equation}\label{eq:shuffleintro}
P_{(m,n)}\cdot 1 = \hik_{(m,n)}(X; q,t).
\end{equation}
 By a beautiful result of Schiffmann-Vasserot \cite{SchiffmannVasserotK,SchiffmannVasserotMac}, the elliptic Hall algebra arises as the inverse limit of Cherednik's \newword{spherical double affine Hecke algebras} of type $\mathfrak{gl}(r)$ $\DAHA{r}$ \cite{Cherednik}, 
 \[
\EHA = \varprojlim_{r \to \infty} \DAHA{r},
\]
under which the generators $P^{(r)}_{(m,n)}$ of $\DAHA{r}$ are mapped precisely to those of $\EHA$. Hence, $\DAHA{r}$ may be thought of as an $r$-variable truncation of the (positive part of) the elliptic Hall algebra $\EHA$.

Our Theorem \ref{thm:main1} connects the higher rank Catalan polynomial $C^{(r)}_{(m,n)}(x_1, \dots, x_r; q,t)$ with $\DAHA{r}$ as follows. The algebra $\DAHA{r}$ comes equipped with a \emph{polynomial} representation $\pol{r}$ on the algebra of symmetric polynomials $\Symr$, defined by Cherednik in terms of Macdonald operators \cite{Cherednik}. Schiffmann-Vasserot show that under the inverse limit above, this polynomial representation induces a \emph{polynomial} representation of $\EHA$ on $\Sym$ that is isomorphic to its geometric representation via the same plethysm that interchanges the Macdonald polynomials $P_\lambda$ with the modified Macdonald polynomials $\tilde{H}_\lambda$ \cite{SchiffmannVasserotK, GorskyNegut}. By defining a `truncated' plethysm on $\Symr$, we show these constructions induce a \emph{geometric} action of $\DAHA{r}$ that yields a truncated Hikita polynomial, in the sense of \eqref{eq:shuffleintro}. Combined with Theorem \ref{thm:main1}, we obtain a \newword{finite Rational Shuffle Theorem}.

\begin{theorem}[Thm. \ref{thm:FiniteShuffle}]\label{thm:main3}
There exists an action of the spherical DAHA $\DAHA{r}$ on $\Symr$ such that, for positive coprime $m$ and $n$ we have:
\[
P^{(r)}_{(m,n)}\cdot 1 = C^{(r)}_{(m,n)}(x_1, \dots, x_r; q,t).
\]
This action is isomorphic to the polynomial representation $\pol{r}$. 
\end{theorem}

It is important to note that explicitly computing the action of Theorem \ref{thm:main3} is a nontrivial task. Indeed, while we understand the precise isomorphism relating it to the polynomial representation of $\DAHA{r}$, the map, which stems from plethystic substitution, is complicated. We plan to make this representation precise in future work. Given the deep connections between the geometric representation of $\EHA$ and important questions in low dimensional topology and algebraic geometry, it is likely similar connections exist in our case. 

\subsection{Structure of the paper} In \S \ref{sec:sspf} we review the basic notions of (semi)standard parking functions, Dyck paths, Catalan numbers, and the Hikita polynomial, as well as prove Theorem \ref{thm:coprimeIntro}. Paving the way to the geometric interpretation of dinv, in \S \ref{sec:affcomp} we introduce affine compositions and prove Theorem \ref{thm:bijectionIntro}, where an explicit construction of this bijection is provided.  In \S \ref{sec:standardization} we introduce the standardization map which we use to give a formula for the $r$-variable truncation of Gessel's quasisymmetric functions in Theorem \ref{thm:gessel vs standardization}, before proving Theorem \ref{thm:main1}. We move to the geometric setting in \S \ref{sec:springer}, where after preliminaries in \S \ref{sec:affflags}--\ref{sect: fixed points} we give a geometric interpretation in \S\ref{sec:geo standardization} before proving  Theorem \ref{thm:main2} in \S\ref{sec:dimension cells}.  In \S \ref{sec:DAHA} we recall the action of the elliptic Hall algebra on the algebra of symmetric functions, as well as the relation between the elliptic Hall algebra and the spherical double affine Hecke algebras, and use this to prove Theorem \ref{thm:main3}. In \S \ref{sec:noncoprime} we find a combinatorial formula, in the spirit of Bizley \cite{bizley}, for $C^{(r)}_{(m,n)}$ in the case where $m$ and $n$ are not coprime. Finally, in Appendix \ref{sec:appendix} we describe how the $\dinv$-statistic can equivalently be defined in terms of laser-rays, showing in particular that for $r = 1$ our $\dinv$ statistic coincides with the usual $\dinv$ statistic on Dyck paths.

\subsection*{Acknowledgments} We are grateful to Sami Assaf, Eugene Gorsky, Oscar Kivinen, Vasily Krylov, Andrei Negu\c{t}, Pablo Ocal, and Jacinta Torres for stimulating conversations and correspondence. M. Vazirani was partially supported
by Simons Foundation Grant 707426.

\section{Semistandard parking functions}\label{sec:sspf}

In this section, after reviewing basic facts about Dyck paths and parking functions, we define a dinv statistic on the set of semistandard parking functions and prove Theorem \ref{thm:main1}. 

Henceforth, unless explicitly stated otherwise, $m$ and $n$ will always denote {\bf coprime} positive integers. 

\subsection{Dyck paths and $(q,t)$-Catalan numbers}\label{ssec:dyck paths}  
Let $R_{(m,n)} \subseteq \R^2$ denote the rectangle with corners $(0,0), (n,0), (0,m)$ and $(n,m)$. The diagonal going from the northwest to the southeast corner of this rectangle has equation $mx + ny - mn = 0$ (see Figure \ref{fig:strip}). We will identify a unit square in the lattice $\Z^2$ by its upper right corner, and refer to these cells as `boxes'.  

For $m,n \in \Z_{>0}$ coprime, define the function $\gamma:\Z^2 \to \Z$ by 
\begin{equation}\label{eq:anderson}
\gamma(x,y):= mn-mx-ny.
\end{equation}
For any box $(x,y) \in \Z^2$, we call the value $\gamma(x,y)$ the \newword{Anderson label} of the box, after Anderson who first studied these labels in \cite{Anderson}.  

\begin{remark}\label{rmk:left diagonal}
Note that a box $(a,b)$ is weakly to the left of the diagonal line $mx + ny - mn = 0$ if and only if $\gamma(a,b) \geq 0$. 
See Figure \ref{fig:strip}.
\end{remark}

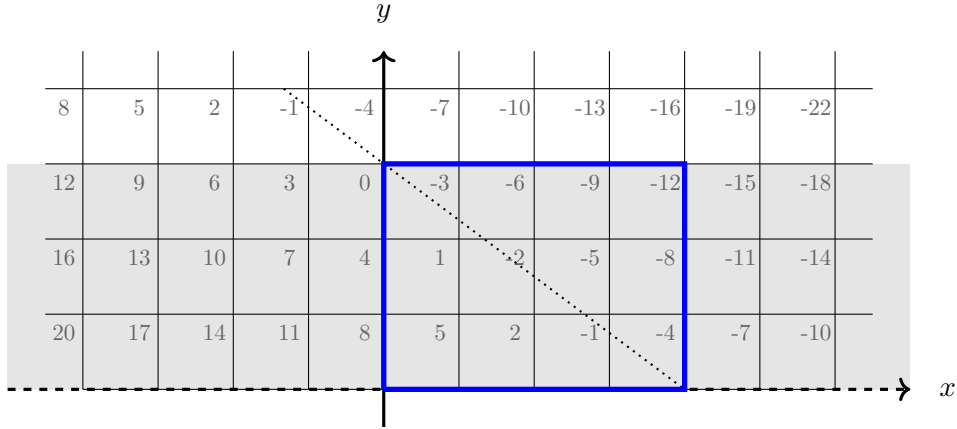
\begin{figure}[ht]
\begin{tikzpicture}

\fill[fill=gray, fill opacity = 0.2] (-1,0)--(-1,3)--(11,3)--(11,0)--cycle;


 \pgfmathtruncatemacro{\m}{4} 
  \pgfmathtruncatemacro{\n}{10} 
 \pgfmathtruncatemacro{\mi}{\m-1} 
    \pgfmathtruncatemacro{\ni}{\n-1} 
  \foreach \x in {0,...,\n} 
    \foreach \y in {1,...,\m}     
       {\pgfmathtruncatemacro{\label}{ - 3 *( \x-4) - 4 *  \y +12} 
       \node  at (\x-.25,\y-.25) [black!60!white, scale=.8]{\label};} 
\draw[ thick, dotted] (2.67,4)--(8,0); 
  \foreach \x in {0,...,\n} 
    \foreach \y  [count=\yi] in {0,...,\mi}   
      \draw (\x,\y)--(\x,\yi);
        \foreach \y in {0,...,\m}     
       \foreach \x  [count=\xi] in {0,...,\ni}   
         \draw (\x,\y)--(\xi,\y);



 \node at (4,5){$y$};
 \node at (11.5,0){$x$};
   \foreach \x in {1,...,4} 
 \draw(-.5,\x)--(0,\x);
 \foreach \x in {1,...,4} 
 \draw(10,\x)--(10.5,\x);
 \foreach \x in {0,...,10} 
 \draw(\x,4)--(\x,4.5);
 \draw[very thick,->] (4,-.5)--(4,4.5);
 \draw[very thick,dashed, ->] (-1,0)--(11,0);

 \draw[line width = 2pt,blue] (4,0)--(8,0)--(8,3)--(4,3)-- (4,0);
\end{tikzpicture}

\caption{The horizontal strip $\strip_{3}$ in gray.
The Anderson labels given by the map $\gamma$ are  marked.
Observe that non-negative labels occur weakly below the dotted diagonal.
 Note that the map $\gamma: \strip_{3} \to \Z$ is bijective.
The rectangle $R_{(3,4)}$ is
outlined in {\color{blue} blue}. 
}
\label{fig:strip}
\end{figure}

Recall that an \newword{$(m,n)$-Dyck path} $\D$ is a lattice path consisting of east and south steps that starts at $(0,m)$, ends at $(n,0)$, and stays below the diagonal line $mx + ny - mn = 0$. We will denote by $\dyck{m,n}$ the set of all $(m,n)$-Dyck paths.

It is well-known, see e.g. \cite{bizley}, that for $\gcd(m,n) = 1$ one has:
\[
|\dyck{m,n}| = C_{(m,n)} = \frac{1}{m+n}\binom{m+n}{m}.
\]

We will consider the following statistics on the set $\dyck{m,n}$. The \emph{area} of $\D \in \dyck{m,n}$ is the number of boxes in $R_{(m,n)}$ completely contained between the path $\D$ and the diagonal $mx + ny - mn = 0$. In light of Remark \ref{rmk:left diagonal}, this statistic is equivalently seen as the number of points $(x,y) \in R_{(m,n)}$ strictly above $\D$ for which $\gamma(x,y) > 0$.  
The other statistic, \emph{dinv} \cite{ALW, Loehr} is more complicated and will be explained in greater generality below.  We set:
\begin{equation}\label{eq:dinv-codinv}
\coarea(\D) := \frac{(m-1)(n-1)}{2} - \area(\D), \qquad \codinv(\D) := \frac{(m-1)(n-1)}{2} - \dinv(\D).
\end{equation}
It is known that $\area(\D), \dinv(\D)$ are non-negative integers.

\begin{definition}\cite{ALW}
The $(m,n)$ \emph{rational $(q,t)$-Catalan number} is the polynomial
\begin{equation}
C_{(m,n)}(q,t) := \sum_{\D\in \dyck{m,n}}q^{\area(\D)}t^{\dinv(\D)}.
\end{equation}
\end{definition}

The polynomial $C_{(m,n)}(q,t)$ presents a number of remarkable features, for example it is $(q,t)$-symmetric and can be written as the following expression of $q$-binomials,
\[
C_{(m,n)}(q,q^{-1}) = \frac{1}{[m+n]_{q}}\qbinom{m+n}{m}.
\]

\subsection{Parking functions and the Hikita polynomial}\label{sec:parking} Clearly, a Dyck path $\D \in \dyck{m,n}$ has precisely $m$-vertical steps. Let us denote by $\ver(\D)=\lbrace v_1, \dots, v_m \rbrace$ the set of vertical steps of the Dyck path $\D$.
\begin{definition}\cite{ALW}
An \emph{$(m,n)$-parking function} is a pair $(\D, \varphi)$ where:
\begin{enumerate}
\item $\D \in \dyck{m,n}$ is an $(m,n)$-Dyck path and
\item $\varphi$ is a bijection $\varphi: \ver(\D) \to \{1, \dots, m\}$,
\end{enumerate}
such that, if $v_i, v_{i+1}$ are two \emph{consecutive} vertical steps with $v_i$ above $v_{i+1}$, then $\varphi(v_i) < \varphi(v_{i+1})$. We will denote by $\park{m,n}$ the set of all $(m,n)$-parking functions. 
\end{definition}

We recall that the \newword{affine symmetric group} is the group of $m$-periodic bijections on $\Z$ under composition:
\[
\AS_m:=\left\{ \sigma: \Z \to \Z \; \vert \; 
\sigma \text{ is a bijection, } \sigma(x+m)=\sigma(x)+m, \text{ and } \sum_{i=1}^m \sigma(i) = {{m+1}\choose{2}}\right\}.
\]
By the periodicity condition, any $\sigma \in \AS_m$ is completely determined by $\sigma(1), \dots, \sigma(m)$. Thus, we will use window notation and write $\sigma = [\sigma(1), \dots, \sigma(m)]$. We say $\sigma \in \AS_m$ is \emph{n-stable} if $\sigma(x+n) > \sigma(x)$ for all $x \in \Z$ and denote by $\AS_m^n$ the set of all $n$-stable affine permutations. 

Parking functions admit analogous statistics to those of Dyck paths. Given any $(\D, \varphi) \in \park{m,n}$, its $\emph{area}$ is just the area of its underlying Dyck path,
\begin{equation}
\area(\D, \varphi) := \area(\D).
\end{equation}
To define the dinv of a parking function, we will use a natural bijection defined originally in \cite{GMV}:
\begin{equation}
\cA: \park{m,n} \to \affsym{m}{n}.
\end{equation}
In \S \ref{sec:affcomp} we will generalize this map to the parabolic setting, so we defer an explicit description of this bijection until then. Assuming the existence of the map $\A$, we define:
\begin{equation}
\codinv(\D, \varphi) := |\{(i, h) \in \{1, \dots, m\} \times \{1, \dots, n\} \mid \omega(i+h) < \omega(i)\}|
\end{equation}
where $\omega := \cA(\D, \varphi)$, so that by \eqref{eq:dinv-codinv},  $\dinv(\D, \varphi) = \frac{(m-1)(n-1)}{2} - \codinv(\D, \varphi)$.   

Let us proceed with the definition of the Hikita polynomial \cite{hikita}, we follow \cite{GM}. Given $\omega \in \AS_m$, recall that a \emph{descent} of $\omega^{-1}$ is an element $j \in \{1, \dots, m-1\}$ such that $\omega^{-1}(j+1) < \omega^{-1}(j)$. We let $\des(\omega^{-1})$ denote the set of descents of $\omega^{-1}$. 

Introduced by Gessel in \cite{Gessel}, the \emph{fundamental quasisymmetric functions} $Q_{S}(X)$ are a family, indexed by finite sets $S$, that form a basis for the space of quasisymmetric functions. 
\begin{equation}\label{eq:gessel}
Q_{S}(X) := \sum_{\substack{i_1 \leq \cdots \leq i_m \\ i_j < i_{j+1} \; \text{if} \; j \in S}}x_{i_{1}}\dots x_{i_{m}}.
\end{equation}

For each coprime pair $m,n \in \Z_{>0}$, the \newword{Hikita polynomial} is the following expression \cite{hikita}:
\begin{equation}\label{eq:Hikita}
\hik_{(m,n)}( X; q,t) := \sum_{(\D, \varphi) \in \park{m,n}}q^{\area(\D, \varphi)}t^{\dinv(\D, \varphi)}Q_{\des(\omega^{-1})}(X),
\end{equation}
where $\omega = \cA(\D, \varphi)$. While these polynomials arise as the Frobenius character of certain representations \cite{hikita, Mellit-rational} and are therefore Schur-positive and hence symmetric in the $x$-variables \cite{GM}, remarkably, they are also $(q,t)$-symmetric. 
The relationship between the Hikita polynomials and the $(q,t)$-Catalan numbers is as follows:
\begin{equation}\label{eq:cat vs hikita}
C_{(m,n)}(q,t) = \langle \hik_{(m,n)}, h_{m}\rangle.
\end{equation}

Let us now denote by $\hik_{(m,n)}^{(r)}$ the truncation to $r$-variables of the Hikita polynomial:
\begin{equation}\label{eq:truncatedHik}
\hik_{(m,n)}^{(r)}(x_1\dots, x_r;q,t) := \hik_{(m,n)}(X;q,t)|_{x_{r+1} = x_{r+2} = \cdots = 0}.
\end{equation}
Then \eqref{eq:cat vs hikita} can equivalently be stated as:
\[
\hik_{(m,n)}^{(1)}(x_1;q,t) = x_{1}^{m}C_{(m,n)}(q,t).
\]

\subsection{Semistandard parking functions} We now semistandardize the definition of an $(m,n)$-parking function. As stated in the introduction, our motivation comes from a connection to the representation theory of quantized Gieseker varieties, cf. \cite{EKLS}. Recall that we have fixed coprime positive integers $m, n$. Let $r > 0$ be a positive integer, that we do not assume to be coprime with $m$ or $n$. 

\begin{definition}[\cite{EKLS}]\label{def:sspf}
A \emph{rank $r$ semistandard $(m,n)$-parking function} is a pair $(\D, \Upsilon)$ where:
\begin{enumerate}
\item $\D \in \dyck{m,n}$ is an $(m,n)$-Dyck path,  and
\item $\Upsilon$ is a function $\Upsilon: \ver(\D) \to \{1, \dots, r\}$,
\end{enumerate}
such that, if $v_i, v_{i+1}$ are two \emph{consecutive} vertical steps with $v_i$ above $v_{i+1}$, then $\Upsilon(v_i) \leq \Upsilon(v_{i+1})$. We will denote by $\sspf{r}{m,n}$ the set of all rank $r$ semistandard $(m,n)$-parking functions. 
\end{definition}

Define the \newword{rank r rational Catalan number} $C^{(r)}_{(m,n)}$ as the cardinality of $\sspf{r}{m,n}$. 
The theorem below was first proven in \cite{EKLS} using representation-theoretic techniques. We provide a new elementary proof.

\begin{theorem}\label{thm:coprime sspf}
Assume $m, n$ are coprime. Then, for any $r$:
\[
C^{(r)}_{(m,n)} = \frac{1}{n}\binom{nr+m-1}{m}.
\]
\end{theorem}

\begin{proof}
 First, note that $\binom{nr+m-1}{m}$ is the cardinality of the set:
\[
\Z^{n\times r}_{\geq 0}(m) := \left\{(a_{11}, \dots, a_{1r}|a_{21}, \dots, a_{2r}| \dots| a_{n1}, \dots, a_{nr}) \in \Z^{nr}_{\geq 0} \mid \sum_{i,j} a_{ij} = m\right\}.
\]
There is an action of the cyclic group $\Z/(n) = \langle \rho \mid \rho^{n} = 1 \rangle$ on $\Z^{n\times r}_{\geq0}(m)$ by a cyclic shift of the first subindex:
\[
\rho\cdot(a_{11}, \dots, a_{1r}|\dots| a_{n1}, \dots, a_{nr}) = (a_{21}, \dots, a_{2r}| \dots| a_{n1}, \dots, a_{nr}| a_{11}, \dots, a_{1r}).
\]
Since $m$ and $n$ are coprime this action is free, so there are precisely $\frac{1}{n}\binom{nr+m-1}{m}$ orbits. Thus, it suffices to identify these orbits with the set $\sspf{r}{m,n}$. For convenience, set $a_{i} := \sum_{j =1}^{r}a_{ij}$, with $i = 1, \dots, n$. 

Now, consider the set $P^{-}_{r}(m,n)$ of pairs $(\D, \Upsilon)$, where $\D$ is a lattice path (not necessarily a Dyck path) inside $R_{(m,n)}$ that ends on an east step, and $\Upsilon$ is a function from the vertical steps of $\D$ to $\{1, \dots, r\}$ that is semistandard in the sense of Definition \ref{def:sspf}. Note that $\sspf{r}{m,n} \subseteq P^{-}_{r}(m,n)$. 

We claim that there is a bijection $\Z_{\geq 0}^{n\times r}(m) \to P^{-}_{r}(m,n)$, sending the tuple $(a_{ij})$ to the pair $(\D, \Upsilon)$ where (see Figure \ref{fig:from tuple to sspf}):
\begin{itemize}
\item $\D$ is the path consisting of $a_1$ south steps, followed by an east step, followed by $a_2$ south steps, followed by an east step, $\dots$, followed by $a_n$ south steps, followed by an east step.
\item $\Upsilon$ is the function that assigns the value $1$ to the first $a_{11}$ south steps, $2$ to the next $a_{12}$ south steps, $\dots$, $r$ to the next $a_{1r}$ south steps, $1$ to the next $a_{21}$ south steps and so on.
\end{itemize}
Upon this identification, it is straightforward to check that each $\Z/(n)$-orbit on $P^{-}_{r}(m,n)$ contains a unique semistandard parking function, from which the result follows. 
\end{proof}

\begin{remark}
For a combinatorial formula for $C^{(r)}_{(m,n)}$ in the case when $m$ and $n$ are not coprime, see Section \ref{sec:noncoprime}. 
\end{remark}

\begin{figure}[ht]
\begin{tikzpicture}[scale=0.7]
 \pgfmathtruncatemacro{\m}{7}
  \pgfmathtruncatemacro{\n}{5}
 \pgfmathtruncatemacro{\mi}{\m-1} 
    \pgfmathtruncatemacro{\ni}{\n-1} 
    %
\draw[dotted] (0,\m)--(\n,0); 
  \foreach \x in {0,...,\n} 
    \foreach \y  [count=\yi] in {0,...,\mi}    
      \draw (\x,\y)--(\x,\yi);
        \foreach \y in {0,...,\m}     
       \foreach \x  [count=\xi] in {0,...,\ni}          
         \draw (\x,\y)--(\xi,\y);
\draw[line width = 1.5] (0,7)--(0,6)--(1,6)--(1,4)--(3,4)--(3,1)--(4,1)--(4,0)--(5,0);
\node[blue] at (0.3, 6.5) [scale=1.3]{2};
\node[red] at (1.3, 5.5) [scale=1.3]{1};
\node[blue] at (1.3, 4.5) [scale=1.3]{2};
\node[red] at (3.3, 3.5) [scale=1.3]{1};
\node[red] at (3.3, 2.5) [scale=1.3]{1};
\node[blue] at (3.3, 1.5) [scale=1.3]{2};
\node[red] at (4.3, 0.5) [scale=1.3]{1};l

\end{tikzpicture}
\qquad
\begin{tikzpicture}[scale=0.7]
 \pgfmathtruncatemacro{\m}{7}
  \pgfmathtruncatemacro{\n}{5}
 \pgfmathtruncatemacro{\mi}{\m-1} 
    \pgfmathtruncatemacro{\ni}{\n-1} 
    %
\draw[dotted] (0,\m)--(\n,0); 
  \foreach \x in {0,...,\n} 
    \foreach \y  [count=\yi] in {0,...,\mi}    
      \draw (\x,\y)--(\x,\yi);
        \foreach \y in {0,...,\m}     
       \foreach \x  [count=\xi] in {0,...,\ni}          
         \draw (\x,\y)--(\xi,\y);
\draw[line width = 1.5] (0,7)--(0,4)--(1,4)--(1,3)--(2,3)--(2,2)--(3,2)--(3,0)--(5,0);
\node[red] at (0.3, 6.5) [scale=1.3]{1};
\node[red] at (0.3, 5.5) [scale=1.3]{1};
\node[blue] at (0.3, 4.5) [scale=1.3]{2};
\node[red] at (1.3, 3.5) [scale=1.3]{1};
\node[blue] at (2.3, 2.5)[scale=1.3] {2};
\node[red] at (3.3, 1.5) [scale=1.3]{1};
\node[blue] at (3.3, 0.5) [scale=1.3]{2};l
\end{tikzpicture}

\caption{The element of $P^{-}_{2}(7,5)$ associated to the tuple
$(0,\blueone|\redone,\blueone|0,0|\redtwo,\blueone|\redone,0)
\in \Z^{5 \times 2}_{\geq 0}(7)$ (left) and
the unique semistandard parking function $(\D, \Upsilon)$ on its
$\Z/(5)$-orbit  corresponding to
$(\redtwo,\blueone|\redone,0|0,\blueone|\redone,\blueone|0,0)$
(right). }
\label{fig:from tuple to sspf}
\end{figure}
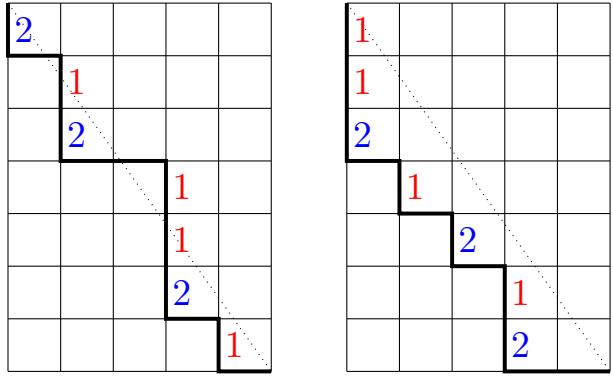

A few easy remarks are in order. First, we have a clear bijection $\dyck{m,n} \to \sspf{1}{m,n}$, that sends a Dyck path $\D$ to the pair $(\D, \Upsilon)$ where $\Upsilon$ is the constant function $1$. Second, $\park{m,n} \subseteq \sspf{m}{m,n}$ where this is a proper inclusion.  Indeed, we have that $\sspf{r}{m,n} \subseteq \sspf{r+1}{m,n}$ for all $r > 0$. 

Our goal now will be to produce a multiparametric deformation of the number $C^{(r)}_{(m,n)}$ that specializes to the usual rational $(q,t)$-Catalan number $\C_{(m,n)}(q,t)$ when $r = 1$. In order to do so, we need certain statistics on the set $\sspf{r}{m,n}$. As before, for $(\D, \Upsilon) \in \sspf{r}{m,n}$ we have:
\begin{equation}
\area(\D, \Upsilon) := \area(\D).
\end{equation}
The \emph{weight} of $(\D, \Upsilon)$ is the $r$-tuple:
\begin{equation}
\we(\D, \Upsilon) := (\we_1, \dots, \we_r) \in \Z_{\geq 0}^{r}, \qquad \we_i(D, \Upsilon) := |\Upsilon^{-1}(i)|.
\end{equation}
Since, $\sum_{i = 1}^{r} \we_i(\D, \Upsilon) = m$, then $\we(\D, \Upsilon)$ is always a (weak) composition of size $m$.  Henceforth, all compositions in this paper will be weak, i.e. some parts may equal zero. We will often use the notation $\bw \models m$ to denote such a composition, and write $\ell(\bw)$ for the number of parts, or \emph{length}, of $\bw$.

If $\bw \in \Z^{r}_{\geq 0}$ is an $r$-part composition of $m$, we denote by
\begin{equation} \label{eq:sspf-nested}
\sspf{\bw}{m,n} \subseteq \sspf{r}{m,n}
\end{equation}
the set of rank $r$ semistandard parking functions of weight $\bw$. For example, for $(1^m):=(1, 1, \dots, 1)$,
\begin{equation} \label{eq:park-nested-sspf}
\park{m,n} = \sspf{(1^m)}{m,n}.
\end{equation}

\subsection{Affine compositions}\label{sec:affcomp} Our next goal is generalize the dinv statistic in the framework of semistandard parking functions. To do so, we will build upon work of Gorsky, Mazin and the third author \cite{GMV} and employ the combinatorics of affine permutations, or in our context, more general objects which we call \emph{affine compositions}. For $ r \in \Z_{>0}$, denote by $[1,r]$ the set $\lbrace 1, \dots, r \rbrace$. 

\begin{definition}\label{def:affcomp}
Let $m, r > 0$. An $(m, r)$-\emph{affine composition} is a function $f: \Z \to \Z$ satisfying the following properties:
\begin{enumerate}
\item $f(x+m) = f(x) + r$ for every $x \in \Z$,
\item $f^{-1}[1, r]$ contains exactly $m$ elements, which are pairwise distinct mod $m$,
\item $\sum_{x \in f^{-1}[1, r]}x = 1 + \dots + m = \frac{m(m+1)}{2}$.
\end{enumerate}
We denote by $\affc{}(m, r)$ the set of $(m, r)$-affine compositions. 
\end{definition}

Notice, that due to Definition \ref{def:affcomp}(1) it follows that $f$ is completely determined by $f(1), \dots, f(m)$. Just as with affine permutations,  we will sometimes utilize window notation for $f$,
\[
f = [f(1), \dots, f(m)]_{r}.
\]
Since the subscript $r$ above is only there to specify the period of $f$ in Definition \ref{def:affcomp}(1), we will often omit it when deemed clear. 

The \emph{weight} of an affine composition $f$ is the $r$-tuple,
\begin{equation}
\we(f) := (|f^{-1}(1)|, \dots, |f^{-1}(r)|).
\end{equation}
Again, by Definition \ref{def:affcomp}(2),  $\we(f)$ is an $r$-part composition of $m$.  We will define by $\affc{\bw}(m,r)$ the set of affine compositions of weight $\bw$. 

Note that when $r=m$ and $\bw = (1^m)$, we recover the affine symmetric group $\affsym{m}{}$. That is, 
\[
\affsym{m}{} = \affc{(1^m)}(m,m). 
\]
Finally, we say that an affine composition $f$ is \newword{n-stable} if $f(x+n) \geq f(x)$ for every $x \in \Z$. We denote by $\affc{\bw}^n(m,r)$ the set of $n$-stable affine compositions of weight $\bw$.

\begin{theorem}\label{thm:bijection} Let $\bw$ be any $r$-part composition of $m$,  with $m,n,r \in \Z_{>0}$ and $m,n$-coprime. 
There exists a weight-preserving bijection
\[
\cA_{\bw}: \sspf{\bw}{m,n} \to \affc{\bw}^n(m,r), 
\]
such that when $r = m$ and $\bw = (1^m)$, the bijection $\cA_{\bw}: \park{m,n} \to \affsym{m}{n}$ coincides with (the inverse of) that defined in \cite[Section 3.1]{GMV}. 
\end{theorem}

\begin{proof}
We will construct the map $\cA_{\bw}$ explicitly. Recall from \eqref{eq:anderson} the Anderson function $\gamma(a,b) = mn - ma - nb $, characterized by $\gamma(0,m) = 0$, $\gamma(a+1, b) = \gamma(a,b) - m$ and $\gamma(a,b+1) = \gamma(a,b) - n$. 
Consider the horizontal strip of height $m$, $\strip_{m} := \{(a,b) \in \Z^2 \mid a \in \Z, 1 \leq b \leq m\}$ (see Figure \ref{fig:strip}). By coprimality of $m$ and $n$, the restriction $\gamma: \strip_{m} \to \Z$ is a bijection. In particular, $R_{(m,n)} \subseteq \strip_{m}$, thus the values of $\gamma$ restricted to $R_{(m,n)}$ are all distinct.  

So suppose $(\D, \Upsilon) \in \sspf{r}{m,n}$. Labeling each $v_i \in \ver(\D)$ by its topmost coordinate, we have that $v_0 = (a_0, m), v_1 = (a_1, m-1), \dots, v_{m-1} = (a_{m-1}, 1)$, for some integers $a_i$ with $a_0 = 0$. Define the function $\tilde{f}: \Z \to \Z$ by setting
\begin{equation}\label{eqn:deftildef}
\tilde{f}(\gamma(v_j)) := \Upsilon(v_j), \text{ for each } j = 0, \dots, m-1
\end{equation}
and extending linearly so that $\tilde{f}$ satisfies (1) of Definition \ref{def:affcomp}. That is, for any $x \not\in \{\gamma(v_j)\}_{0 \leq j \leq m-1}$, we can write $x = \gamma(v_j)+ pm$ for some $p \in \Z$ and $j \in [0,m-1]$ so that $\tilde{f}(x) := \tilde{f}(\gamma(v_j)) + pr$. 

To see that the function $\tilde{f}$ is well-defined, note that we may extend the function $\Upsilon: \ver(\D) \to \{1, \dots, r\}$ to a function $\Upsilon: \strip_{m} \to \Z$ by setting
\[
\Upsilon(a_{j} + k, m-j) := \Upsilon(a_{j}, m-j) - kr, \text{ for all } j = 0, \dots, m-1, \text{ and } k \in \Z
\]
so that \eqref{eqn:deftildef} is valid for every $i \in \Z$. Since $\Upsilon$ is single-valued and $\gamma: \strip_{m} \to \Z$ is bijective, it follows that $\tilde{f}$ is well-defined. 

We claim that $\tilde{f}$ satisfies (1) and (2), but possibly not (3), of Definition \ref{def:affcomp}. For (1), if $\tilde{f}(x) = i$ then, assuming that $x$ corresponds to the box $(a,b)$, $\tilde{f}(x+m) =\Upsilon(a-1,b) = f(x) + r$.  For (2), by definition $\tilde{f}^{-1}([1,r]) = \{\gamma(v_0), \dots, \gamma(v_{m-1})\}$. Since the $y$-coordinates of $v_j =( a_{m-j},j)$ run from $1$ to $m$ and $m, n$ are coprime, it follows that $\gamma(v_0), \dots, \gamma(v_{m-1})$ are all distinct mod $m$. 

Now, $\sum_{i= 0}^{m-1}\gamma(v_{i}) - (i+1) = km$ for some integer $k$. Hence, we can define the function:
\begin{equation}\label{eq:def A-map}
\cA_{\bw}(\D, \Upsilon) := f, \; \text{where} \; f(x) := \tilde{f}(x+k).
\end{equation}
In particular, $f$ is an $(m,r)$-affine composition. By construction, it is clear that $\we(\D, \Upsilon) = \we(f)$. Finally, the fact that $f$ is $n$-stable follows easily from the semi-standardness of $(\D, \Upsilon)$. 

To see the map is a bijection we explicitly construct the inverse map $\cA_{\bw}^{-1}$. Suppose $f \in \affc{}(m, r)$ and consider the set $\Delta_{f} := \{d \in \Z \mid f(d) \geq 0\}$. Then $\Delta_{f} $ is bounded below and hence has a minimal element, say $M_{f}$. Let $\gamma_{f}: \strip_{m} \to \Z$ be the shifted Anderson function given by $\gamma_{f}(a,b) = \gamma(a,b)+ M_{f}$. We claim that the set of all boxes $B_{f} := \{(a,b) \in R_{(m,n)} \mid \gamma_{f}(a,b) \in \Delta_{f}\}$ form the coarea boxes of a Dyck path. 
From the properties of $\gamma$ mentioned above and Remark \ref{rmk:left diagonal},  it is obvious that $\gamma_f$ is bijective on $\strip_{m}$ and that $\gamma_{f}(a,b) \geq M_{f}$ if and only if the box $(a,b)$ is to the left of the diagonal $mx+ny-mn = 0$. Thus, by definition all the boxes in $B_{f} $ lie to the left of the diagonal. Now, if $(a,b) \in B_{f}$ is such that $(a-1, b) \in R_{(m,n)}$, then $f(\gamma_{f}(a-1,b)) = f(\gamma_f(a,b) + m) = f(\ell(a,b)) + r > 0$. And if $(a,b) \in B_{f}$ is such that $(a, b-1) \in R_{(m,n)}$ then $f(\gamma_f(a, b-1)) = f(\gamma_f(a,b) + n) \geq f(\gamma_f(a,b)) \geq 0$. Thus, the boxes in $B_{f}$ are precisely the coarea boxes of a Dyck path $D_f$. 

Lastly, the function $\Upsilon$ can be recovered from $f$ by setting $\Upsilon(a,b) = f(\gamma_f(a,b))$. That $(\D,\Upsilon)$ is semistandard follows easily from the $n$-stability of $f$. From here, it is easy to see that $\cA_{\bw}(\D,\Upsilon)$ recovers $f$. 

In the special case when $\bw = (1^m)$, the equivalence between $\cA_{\bw}^{-1}$ and the bijection in \cite[Section 3.1]{GMV} follows easily from direct comparison. 
\end{proof}

A few important remarks are in order. 

\begin{remark} 
First, for easy referencing, we provide an explicit algorithm for computing the bijection $\cA_{\bw}: \sspf{\bw}{m,n} \to \affc{\bw}^n(m,r)$ in Theorem \ref{thm:bijection}.

Recall the map $\gamma(x,y) = mn-mx-ny$ from \eqref{eq:anderson}. Then, given any $(\D, \Upsilon) \in \sspf{\bw}{m,n}$, the the affine composition ${f}_{(\D, \Upsilon)} := \cA_{\bw}(\D, \Upsilon)$ can be calculated as follows: 
\begin{itemize}
\item[\bf{1:}] 
Set $\tilde{f}(\gamma(v)) = \Upsilon(v)$ for every vertical step $v \in \ver(\D)$,
so that  $\tilde{f}^{-1}(i) = \gamma(\Upsilon^{-1}(i))$ for  $i \in [r]$.

\item[\bf{2:}]  For any $x \not\in \{\gamma(v) \mid v \in  \ver(\D)\}$, write $x = \gamma(v)+ pm$ for some $p \in \Z$ and $v \in \ver(\D)$. Set $\tilde{f}(x) := \tilde{f}(\gamma(v)) + pr$. 

\item[\bf{3:}]  Let $k= \frac{n(m-1)-(m+1)}{2} - \sum_{j=0}^{m-1} a_j$, where $a_{0}, \dots, a_{m-1}$ are the $x$-coordinates of the vertical steps of $\D$. Then $f_{(\D, \Upsilon)}(x) := \tilde{f}(x+k)$ for all $x \in \Z$. 
\end{itemize}

For a detailed example implementing the above algorithm to 
compute  $\A_{\bw}$, see \cite[Example 7]{GSV-fpsac}.
\end{remark}

\begin{remark}\label{rmk:notation}
Second, when $\bw = (1^m)$, we will omit the subscript $\bw$ in $\cA_{\bw}$ and simply denote by $\cA$ the map from $\park{m,n} \to \affsym{m}{n}$. 
\end{remark}

\begin{remark}\label{rmk:recovering pi}
Thirdly, let $(\D, \Upsilon) \in \sspf{\bw}{m,n}$, and $f:=\cA_{\bw}(\D, \Upsilon)$. Note that we can recover the Dyck path $\D$ from the set $f^{-1}[1,r]$ as follows.  The boxes to the left of the vertical steps of $\D$ are precisely $\gamma^{-1}\tilde{f}^{-1}[1,r]$ where, up to a shift, $\tilde{f}$ coincides with $f$, see \eqref{eq:def A-map}. The shift from $f$ to $\tilde{f}$ is done so that the minimal element of $\tilde{f}^{-1}[1,r]$ is $0$. In other words, $\tilde{f}^{-1}[1,r] := \{x - x_{0} \mid x \in f^{-1}[1,r], x_0 := \min f^{-1}[1,r]\}$, and we see that $\D$ is completely determined by the set $f^{-1}[1,r]$. 
\end{remark}

\subsection{Standardization}\label{sec:standardization} Note that we have a right action of $\affsym{m}{}$ on $\affc{}(m, r)$ by precomposition. It is clear that two affine compositions $f, g \in \affc{}(m, r)$ are in the same $\affsym{m}{}$-orbit if and only if $\we(f) = \we(g)$. For each weight $\bw = (w_1, \dots, w_r)$, we consider a distinguished affine composition:
\begin{equation}
f_{\bw} := [1, \dots, 1, 2, \dots, 2, \dots, r, \dots, r]_{r}
\end{equation}
where each letter $i$ repeats $w_{i}$-times. In other words, $f_{\bw}^{-1}(i) = [w_{1} + \cdots + w_{i-1}+1, w_{1} + \cdots + w_{i-1} + w_{i}]$. The stabilizer of $f$ is clearly the parabolic subgroup $S_{\bw} := S_{w_{1}} \times \cdots \times S_{w_{r}}$. Thus, the map
\begin{equation}
\action_{\bw}: \minlength{\bw}{m}{} \to \affc{\bw}(m, r), \qquad \sigma \mapsto f_{\bw}\circ \sigma
\end{equation}
is a bijection. We will abuse the notation and denote by $S_{\bw}\backslash\affsym{m}{}$ a set of minimal-length right coset representatives. In other words, $\sigma \in \minlength{\bw}{m}{}$ if and only if (see e.g. \cite{papi})
\[
\sigma^{-1}(w_1 + \cdots + w_{i-1}+1) < \cdots < \sigma^{-1}(w_1 + \cdots + w_{i-1}+w_i) \; \text{for every} \; i \in \{1, \dots, r\}.
\]

\begin{figure}[ht]
\begin{tikzpicture}[scale=0.7]
 \pgfmathtruncatemacro{\m}{5} 
  \pgfmathtruncatemacro{\n}{7} 
 \pgfmathtruncatemacro{\mi}{\m-1} 
    \pgfmathtruncatemacro{\ni}{\n-1} 
\draw[dotted] (0,\m)--(\n,0); 
  \foreach \x in {0,...,\n} 
    \foreach \y  [count=\yi] in {0,...,\mi}   
      \draw (\x,\y)--(\x,\yi);
        \foreach \y in {0,...,\m}     
       \foreach \x  [count=\xi] in {0,...,\ni}   
         \draw (\x,\y)--(\xi,\y);
  \foreach \x in {0,...,\n} 
  {  \foreach \y in {1,...,\m}     
       {\pgfmathtruncatemacro{\label}{ - \m * \x - \n *  \y +\m*\n} 
      \pgfmathscalar{\label}
       \ifthenelse{ \label >0  \OR \label =0 }{ \node  at (\x-.25,\y-.25) [black!60!white, scale=.7]{\label};}{}} }
 \draw[line width = 1.5](0,5)--(0,3)--(2,3)--(2,1)--(4,1)--(4,0)--(7,0);
 \node at (0.3, 4.5) [scale=1.3]{\color{blue} 2};
\node at (0.3, 3.5) [scale=1.3]{\color{blue} 2};
\node at (2.3, 2.5) [scale=1.3]{\color{blue} 1};
\node at (2.3, 1.5) [scale=1.3]{\color{blue} 3};
\node at (4.3, 0.5) [scale=1.3]{\color{blue} 3};
\end{tikzpicture}
\qquad \qquad
\begin{tikzpicture}[scale=0.7]
\draw(9,0)--(16,0)--(16,5)--(9,5)--cycle;
\draw(9,1)--(16,1);
\draw(9,2)--(16,2);
\draw(9,3)--(16,3);
\draw(9,4)--(16,4);
\draw(10,0)--(10,5);
\draw(11,0)--(11,5);
\draw(12,0)--(12,5);
\draw(13,0)--(13,5);
\draw(14,0)--(14,5);
\draw(15,0)--(15,5);
\draw[dotted](16,0)--(9,5);
\draw[line width = 1.5](9,5)--(9,3)--(11,3)--(11,1)--(13,1)--(13,0)--(16,0);

\node at (9.3, 4.5) [scale=1.3]{\color{red} 2};
\node at (9.3, 3.5) [scale=1.3]{\color{red} 3};
\node at (11.3, 2.5) [scale=1.3]{\color{red} 1};
\node at (11.3, 1.5) [scale=1.3]{\color{red} 5};
\node at (13.3, 0.5) [scale=1.3]{\color{red} 4};
\end{tikzpicture}
\caption{A semistandard parking function (left) and its standardization (right). 
On the left, the values of $\gamma$ are denoted by small grey numbers, with the values of $\Upsilon$ denoted by large {\color{blue}blue} integers to the right of each vertical step. The values of the standardization are denoted in {\color{red}red} labels (right). In particular, the last two vertical steps with label $3$ (left) and Anderson labels $11$ and $8$ correspond, under standardization, to the vertical steps labeled by $4$ and $5$ (right). This is inline with Lemma \ref{lem:standardization} where Anderson labels are used to break ties. }
\label{fig:standardization}
\end{figure}
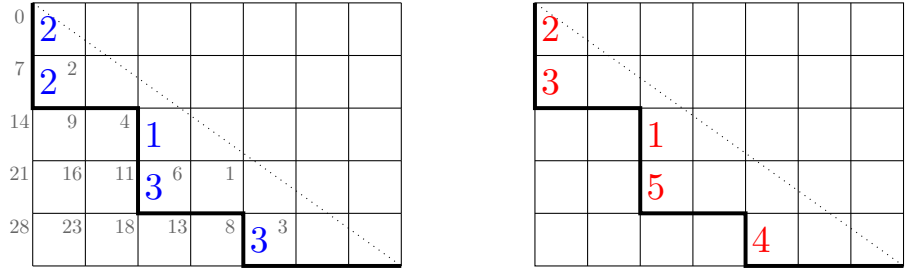

\begin{remark}\label{rmk:increasing}
Note that $\minlength{\bw}{m}{}$ can also be described as those affine permutations $\sigma \in \affsym{m}{}$ for which $\sigma^{-1}$ is increasing on $f_{\bw}^{-1}(y)$ for every $y \in \Z$. 
\end{remark}

Let us recall that if $f: \Z \to \Z$ is a function, then an \emph{inversion} of $f$ is a pair $(i,j) \in \Z \times \Z$ such that $i < j$ and $f(j) < f(i)$. We will denote by $\Inv(f)$ the set of inversions of the function $f$. 

\begin{lemma}\label{lem:inversions}
For every $\sigma \in \minlength{\bw}{m}{}$,
\[
\Inv(\sigma) = \Inv(f_{\bw}\circ\sigma).
\]
\end{lemma}
\begin{proof}
Let $(i,j) \in \Inv(\sigma)$, so $i < j$ but $\sigma(j) < \sigma(i)$. Since $f_{\bw}$ is non-decreasing, we have that $f_{\bw}\circ \sigma(j) \leq f_{\bw}\circ\sigma(i)$. We need to show that this inequality is strict. If not, then we have $\sigma(j) < \sigma(i)$ and both $\sigma(i), \sigma(j) \in f_{\bw}^{-1}(y)$ for some $y \in \Z$. However, by Remark \ref{rmk:increasing} this contradicts the fact that $\sigma^{-1}$ is increasing on $f^{-1}(y)$. Thus, $(i,j) \in \Inv(\sigma)$. The other inclusion follows from the non-decreasingness of $f_{\bw}$. 
\end{proof}

Let us now denote by $\minlength{\bw}{m}{n}$ those minimal length coset representatives $\sigma \in \minlength{\bw}{m}{}$ that do not have inversions of height $n$, that is, such that $f(x) \leq f(x+n)$ for every $x \in \Z$.  Hence, we obtain the following Corollary to Lemma \ref{lem:inversions}. 

\begin{corollary}\label{prop:heightn}
For every $n > 0$, the map $\action_{\bw}$ induces a bijection:
\[
\action_{\bw}:\minlength{\bw}{m}{n} \to \affc{\bw}^n(m, r).
\]
\end{corollary}

This allows us to define the \emph{standardization} of a semistandard parking function $(\D, \Upsilon)$. Fixing a weight $\bw$, we consider the following composition,
\[
\sspf{\bw}{m,n}\buildrel \cA_{\bw} \over \longrightarrow \affc{\bw}^n(m,r) \buildrel \action_{\bw}^{-1}\over\longrightarrow \minlength{\bw}{m}{n}.
\]

Since every element in $\minlength{\bw}{m}{n}$ is, in particular, an affine permutation without $n$-inversions, we can apply the map $\cA^{-1}: \affsym{m}{n} \to \park{m,n}$ (see Remark \ref{rmk:notation}). Thus, we may define the following. 

\begin{definition}
Let $(\D, \Upsilon)$ be a rank $r$ semistandard $(m,n)$-parking function of weight $\bw$. The \newword{standardization} of $(\D, \Upsilon)$ is the $(m,n)$-parking function,
\begin{equation}
\std(\D, \Upsilon) := \cA^{-1}\action_{\bw}^{-1}\cA_{\bw}(D, \Upsilon) \in \park{m,n}.
\end{equation}
\end{definition}

\begin{remark} \label{rem:PFstdtoPF}
Note that when $\bw = (1, \dots, 1)$, then both $f_{\bw}$ and $\action_{\bw}$ are the identity function, thus the standardization of a parking function is itself.
\end{remark}

We would like to obtain a way to go from a semistandard parking function $(\D, \Upsilon)$ to its standardization $(\D', \varphi)$ without having to pass by the several bijections appearing in the definition.  Lemma \ref{lem:standardization} below achieves that. We begin by verifying that the underlying Dyck path of a semistandard parking function coincides with that of its standardization. 

\begin{proposition}\label{prop:same dyck}
Let $(\D, \Upsilon) \in \sspf{\bw}{m,n}$ and $(\D', \varphi) = \std(\D, \Upsilon)$ be its standardization. Then, $\D = \D'$. That is, a semistandard parking function and its standardization have the same underlying Dyck path. 
\end{proposition}
\begin{proof}
Recall from Remark \ref{rmk:recovering pi} that if $(\D, \Upsilon)$ is a semistandard parking function then the underlying Dyck path is recovered directly from the data of $f^{-1}[1, r]$, where $f = \cA_{\bw}(D, \Upsilon)$. Thus, we need to show that for every $\sigma \in \minlength{\bw}{m}{n}$,
\[
\sigma^{-1}[1, m] = (f_{\bw}\circ\sigma)^{-1}[1, r] = \sigma^{-1}f_{\bw}^{-1}[1,r].
\]
By definition, $f_{\bw}^{-1}[1,r] = [1,m]$, so the result follows. 
\end{proof}

In fact, the following is an easy, pictorial way to directly obtain the standardization of a semistandard parking function $(\D, \Upsilon)$ via the Anderson labels given by $\gamma: \Z^2 \to \Z$ from \eqref{eq:anderson}.  See Figure \ref{fig:standardization} for an illustration of this procedure. 

\begin{lemma}\label{lem:standardization}
Let $(\D, \Upsilon)$ be a semistandard parking function. Then, $\std(\D, \Upsilon) = (\D, \varphi)$, where $\varphi: \ver(\D) \to \{1, \dots, m\}$ is the unique bijection satisfying:
\begin{enumerate}
\item $\varphi(v_i) < \varphi(v_j)$ if $\Upsilon(v_i) < \Upsilon(v_j)$, and 
\item $\varphi(v_i) < \varphi(v_j)$ whenever $\Upsilon(v_i) = \Upsilon(v_j)$ and $\gamma(v_i) < \gamma(v_j)$. 
\end{enumerate}
\end{lemma}

\begin{proof}
Let $\bw := \wt(\D, \Upsilon)$, $f := \cA_{\bw}(\D, \Upsilon)$ and $\sigma := \cA(\D, \varphi)$, where $(\D, \varphi)$ is as described in the statement of the lemma. We need to check that
\[
f = f_{\bw}\circ \sigma
\]
and that $\sigma$ has minimal length with this property. Since both $(\D, \Upsilon)$ and $(\D, \varphi)$ have the same underlying Dyck path, by Remark \ref{rmk:recovering pi} it is enough to check that $\tilde{f} = f_{\bw}\tilde{\sigma}$. Now, for each $i = 1, \dots, r$, let $\Upsilon^{-1}(i) = \{v^{(i)}_{1}, \dots, v^{(i)}_{w_{i}}\}$ and $\gamma^{(i)}_{k} := \gamma(v^{(i)}_{k})$, with $k = 1, \dots, w_{i}$. We may and will assume that $\gamma^{(i)}_{1} < \gamma^{(i)}_{2} < \cdots < \gamma^{(i)}_{w_{i}}$. We have:
\[
\tilde{f}^{-1}(i) = \{\gamma^{(i)}_{1} < \gamma^{(i)}_{2} < \cdots < \gamma^{(i)}_{w_{i}}\}.
\]
Now, Condition (1) in the statement of the Lemma is equivalent to saying that $\varphi^{-1}\{1, \dots, w_{1}\} = \{v^{(1)}_{1}, \dots, v^{(1)}_{w_{1}}\}$, $\varphi^{-1}\{w_1 + 1, \dots, w_1 + w_2\} = \{v^{(2)}_{1}, \dots, v^{(2)}_{w_{2}}\}$, etc. From here, it follows that for each $i \in \{1, \dots, r\}$:
\[
(f_{\bw}\tilde{\sigma})^{-1}(i) = \tilde{\sigma}^{-1}\{w_{1} + \cdots + w_{i-1}+1, \dots, w_{1} + \cdots + w_{i-1} + w_{i}\} = \gamma\{v_{1}^{(i)}, \dots, v_{w_{i}}^{(i)}\} = \tilde{f}^{-1}(i).
\]

However, the fact that both $\tilde{f}$ and $f_{\bw}\tilde{\sigma}$ satisfy the same periodicity condition implies that this equality remains true for every $i \in \Z$, which means that $\tilde{f} = f_{\bw}\tilde{\sigma}$ and thus $f = f_{\bw}\sigma$.

To check that $\sigma$ has minimal length, we need to verify that $\sigma^{-1}$ is strictly increasing on $f^{-1}(y)$ for every $y \in \Z$ or, equivalently, for $y = 1, \dots, r$. It is again enough to check this property for $\tilde{\sigma}$ and $\tilde{f}$, which is easily seen to be equivalent to Condition (2) of the Lemma. 
\end{proof}

It is natural question to ask what the corresponding \emph{destandardization} map is. In particular, given any $(\D, \varphi) \in \park{m,n}$ and $r \in \Z_{>0}$ is there always a function $(\D, \Upsilon) \in \sspf{r}{m,n}$ that standardizes to $(\D, \varphi)$? Indeed, for $r \geq m$ since by \eqref{eq:sspf-nested} and \eqref{eq:park-nested-sspf} $\park{m,n} = \sspf{(1^m)}{m,n} \subset \sspf{m}{m,n} \subseteq \sspf{r}{m,n}$ and by Remark \ref{rem:PFstdtoPF} parking functions standardize to themselves, then one can set $\Upsilon = \varphi$.  When $r<m$ however, there exist parking functions $(\D, \varphi) \in \park{m,n}$ for which no $(\D, \Upsilon) \in \sspf{r}{m,n}$ will standardize to $(\D, \varphi)$ (e.g. take $m=4,n=3,r=2$ and consider $\cA^{-1}([-1,1,4,6])$).  

More generally, let us define the $\bw$-semistandardization of a parking function $(\D, \varphi)$ to be the semistandard parking function with the same underlying Dyck path but with $\Upsilon(\varphi^{-1}(i)) = j$ if $i \in [w_{1} + \cdots + w_{j-1} + 1, w_{1} + \cdots + w_{j}]$. Pictorially, $(\D, \Upsilon)$ is obtained by replacing the labels $1, 2, \dots, w_1$ of the parking function by $1$, the labels $w_1 + 1, \dots, w_1 + w_2$ by $2$, and so on. Let us denote $(\D, \Upsilon) =: \sstd{\bw}(D, \varphi)$. By translating through the Anderson labels \eqref{eq:anderson} we may also define the semistandardization of an $n$-stable affine permutation, $\sstd{\bw}(\sigma) \in \affcz{\bw}(m,r)^{n}$.

Note that $\bw$-semistandardization is a left inverse to standardization: if $(\D, \Upsilon)$ is a semistandard parking function of weight $\bw$, then $\sstd{\bw}(\std(\D, \Upsilon)) = (\D, \Upsilon)$. It is not true, however, that $\std(\sstd{\bw}(\D, \varphi)) = (\D, \varphi)$ for a parking function $(\D, \varphi)$. See Figure \ref{fig:semistandardization} below. 

Finally, note that we have: 

\begin{equation}\label{eq:PF=SSPF}
\std_r^{-1}(\park{m,n})= \bigcup_{\substack{\bw \models m \\ \ell(\bw)=r}} \sspf{\bw}{m,n}= \sspf{r}{m,n}.
\end{equation}

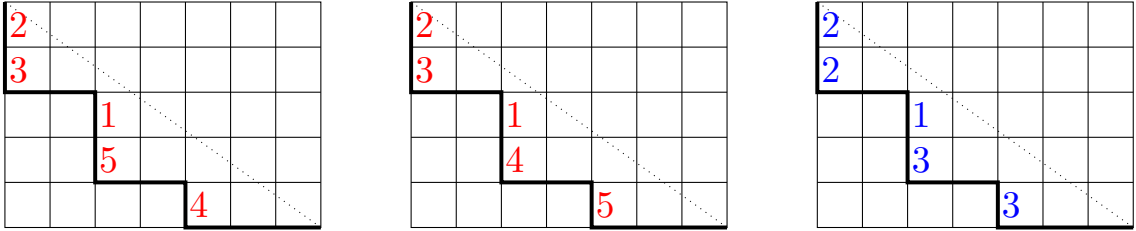
\begin{figure}[ht]
\begin{tikzpicture}[scale=0.6]
\draw(0,0)--(7,0)--(7,5)--(0,5)--cycle;
\draw(0,1)--(7,1);
\draw(0,2)--(7,2);
\draw(0,3)--(7,3);
\draw(0,4)--(7,4);
\draw(1,0)--(1,5);
\draw(2,0)--(2,5);
\draw(3,0)--(3,5);
\draw(4,0)--(4,5);
\draw(5,0)--(5,5);
\draw(6,0)--(6,5);
\draw[dotted](7,0)--(0,5);
\draw[line width = 1.5](0,5)--(0,3)--(2,3)--(2,1)--(4,1)--(4,0)--(7,0);
\node at (0.3, 4.5) [scale=1.3]{\color{red} 2};
\node at (0.3, 3.5) [scale=1.3]{\color{red} 3};
\node at (2.3, 2.5) [scale=1.3]{\color{red} 1};
\node at (2.3, 1.5) [scale=1.3]{\color{red} 5};
\node at (4.3, 0.5) [scale=1.3]{\color{red} 4};

\draw(18,0)--(25,0)--(25,5)--(18,5)--cycle;
\draw(18,1)--(25,1);
\draw(18,2)--(25,2);
\draw(18,3)--(25,3);
\draw(18,4)--(25,4);
\draw(19,0)--(19,5);
\draw(20,0)--(20,5);
\draw(21,0)--(21,5);
\draw(22,0)--(22,5);
\draw(23,0)--(23,5);
\draw(24,0)--(24,5);
\draw[dotted](25,0)--(18,5);
\draw[line width = 1.5](18,5)--(18,3)--(20,3)--(20,1)--(22,1)--(22,0)--(25,0);

\node at (18.3, 4.5) [scale=1.3]{\color{blue} 2};
\node at (18.3, 3.5) [scale=1.3]{\color{blue} 2};
\node at (20.3, 2.5)[scale=1.3] {\color{blue} 1};
\node at (20.3, 1.5) [scale=1.3]{\color{blue} 3};
\node at (22.3, 0.5)[scale=1.3] {\color{blue} 3};

\draw(9,0)--(16,0)--(16,5)--(9,5)--cycle;
\draw(9,1)--(16,1);
\draw(9,2)--(16,2);
\draw(9,3)--(16,3);
\draw(9,4)--(16,4);
\draw(10,0)--(10,5);
\draw(11,0)--(11,5);
\draw(12,0)--(12,5);
\draw(13,0)--(13,5);
\draw(14,0)--(14,5);
\draw(15,0)--(15,5);
\draw[dotted](16,0)--(9,5);
\draw[line width = 1.5](9,5)--(9,3)--(11,3)--(11,1)--(13,1)--(13,0)--(16,0);

\node at (9.3, 4.5) [scale=1.3]{\color{red} 2};
\node at (9.3, 3.5) [scale=1.3]{\color{red} 3};
\node at (11.3, 2.5) [scale=1.3]{\color{red} 1};
\node at (11.3, 1.5) [scale=1.3]{\color{red} 4};
\node at (13.3, 0.5) [scale=1.3]{\color{red} 5};
\end{tikzpicture}
\caption{For $\bw = (1,2,2)$, the preimage of the semistandard parking function on the right under $\bw$-semistandardization consists precisely of the parking functions in the center and left. In particular, only the leftmost parking function is the standardization. }
\label{fig:semistandardization}
\end{figure}

\subsection{The dinv statistic and higher rank $(q,t)$-Catalan polynomials}\label{sec:proofmain1}
\begin{definition}\label{def:dinv}
Let $(\D, \Upsilon)$ be a rank $r$ semistandard $(m,n)$-parking function and $\std(D, \Upsilon) \in \park{m,n}$ its standardization.  Define, 
\begin{equation}\label{eq:dinv}
\dinv(\D, \Upsilon) := \dinv(\std(\D, \Upsilon)).
\end{equation}
\end{definition}

Since $(m,n)$-parking functions standardize to themselves, \eqref{eq:dinv} clearly recovers $\dinv$ for $\park{m,n}$. In Appendix \ref{sec:appendix} we give an equivalent description of $\dinv$ for $\sspf{r}{m,n}$ in terms of laser rays which naturally recovers the $\dinv$ statistic for Dyck paths.

As we will see in Section \ref{sec:springer}, the set of rank $r$ semistandard parking functions label the affine spaces in an affine paving of a parabolic affine Springer fiber,  and the co-dinv statistic is the dimension of the corresponding affine space. We are finally ready to define the higher rank $(q,t)$-Catalan polynomial.

\begin{definition}
The \newword{higher rank $(q,t)$-Catalan polynomials} are the family of polynomials, 
\begin{equation}
C^{(r)}_{(m,n)}(x_1, \dots, x_r;q,t) := \sum_{(\D, \Upsilon) \in \sspf{r}{m,n}}q^{\area(\D,\Upsilon)}t^{\dinv(\D, \Upsilon)}x^{\we(\D, \Upsilon)}, 
\end{equation}
where $m, n, r$ are positive integers and $\gcd(m,n) = 1$. 
\end{definition}

To obtain more information about $C^{(r)}_{(m,n)}(x_1, \dots, x_r;q,t)$, we relate the Gessel quasi-symmetric functions of Section \ref{sec:parking} to the standardization procedure of Section \ref{sec:standardization}. Let us denote by $Q_{S}^{(r)}$ the truncation to $r$-variables of $Q_{S}$. In other words:
\begin{equation}
Q^{(r)}_{S}(X) := \sum_{\substack{1 \leq i_1 \leq \cdots \leq i_m \leq r \\ i_j < i_{j+1} \; \text{if} \; j \in S}}x_{i_{1}}\dots x_{i_{m}}.
\end{equation}

\begin{theorem}\label{thm:gessel vs standardization}
Let $(\D, \varphi)$ be an $(m,n)$-parking function, and let $\std^{-1}_{r}(\D, \varphi)$ be the set of rank $r$-semistandard $(m,n)$-parking functions that standardize to $(\D, \varphi)$. Let also $\omega = \mathcal{A}(\D, \varphi) \in \affsym{m}{n}$. Then,
\[
Q^{(r)}_{\des(\omega^{-1})}(X) = \sum_{(\D, \Upsilon) \in \std^{-1}_{r}(\D, \varphi)}x^{\we(\D, \Upsilon)}.
\]
\end{theorem}

\begin{proof}
Let us first set some notation. Given an $(m,n)$-parking function, denote:
\[
\gamma_{i} := \gamma(\varphi^{-1}(i)) \in \Z, i = 1, \dots, m, \qquad \omega := \mathcal{A}(\D, \varphi) \in \affsym{m}{n}.
\]
We claim that, for $i = 1, \dots, m-1$, $i \in \des(\omega^{-1})$ if and only if $\gamma_{i+1} < \gamma_{i}$. Indeed, recall that $\omega$ is defined by first defining $\tilde{\omega}$ and then shifting. Moreover, $\tilde{\omega}$ is defined by its inverse: $\tilde{\omega}^{-1}(i) = \gamma_{i}$ so that $\omega^{-1}(i) = \gamma_{i} + k$ for some $k$ which is independent of $i$. From here, the claim follows easily.

Now suppose $(\D, \Upsilon)$ is any semistandard parking function that standardizes to $(\D, \varphi)$. Set $a_{i} := \Upsilon(\varphi^{-1}(i))$ to be the values of the vertical steps as labeled by $\Upsilon$. By Lemma \ref{lem:standardization}, necessarily
\[
a_{1} \leq a_{2} \leq \cdots \leq a_{m}.
\]
Thus, by the first claim above it suffices to show show that $a_{i} < a_{i+1}$ if $i$ is in the descent set of $\omega^{-1}$. Assume on the contrary that $a_{i} = a_{i+1}$ with $i \in \des(\omega^{-1})$. Note that we may assume that $i \neq m$ (otherwise the statement $a_{i} = a_{i+1}$ does not make sense). 
Then by the claim above,  $\gamma_{i+1} < \gamma_{i}$.
 Once again, by Lemma \ref{lem:standardization}, this implies that $i+1 = \varphi(\varphi^{-1}({i+1})) < \varphi(\varphi^{-1}({i})) = i$, a contradiction. The result follows.
\end{proof}

Recall the truncation of the Hikita polynomial \eqref{eq:truncatedHik}, 
\[
\hik_{(m,n)}^{(r)}(x_1\dots, x_r;q,t) := \hik_{(m,n)}(X;q,t)|_{x_{r+1} = x_{r+2} = \cdots = 0}.
\]

\begin{theorem}\label{thm:Cat=Hik}
For any positive integers $m,n,r$ with $m,n$ coprime, we have
\[
C^{(r)}_{(m,n)}(x_1, \dots, x_r; q,t) = \hik_{(m,n)}^{(r)}(x_1\dots, x_r;q,t).
\]
In particular, $C^{(r)}_{(m,n)}(x_1, \dots, x_r; q,t)$ is symmetric in $q,t$, as well as symmetric in $x_1, \dots, x_r$. 
\end{theorem}

\begin{proof}
The $q,t$ symmetry and $x_1, \cdots, x_r$ symmetry are inherited under truncation. The equality is an easy consequence of Theorem \ref{thm:gessel vs standardization}. Namely, from \eqref{eq:PF=SSPF} we have that
\[ \bigcup_{(\D,\varphi) \in \park{m,n}} \lbrace (\D,\Upsilon) \in \sspf{r}{m,n} \vert \std(\D, \Upsilon) = (\D,\varphi) \rbrace  = \sspf{r}{m,n},  \]
and since both $\area$ and $\dinv$ are preserved under standardization,  then the following equalities hold:
\begin{align*}
\hik_{(m,n)}^{(r)}(x_1\dots, x_r;q,t) 
&= \sum_{(\D, \varphi) \in \park{m,n}}q^{\area(\D, \varphi)}t^{\dinv(\D, \varphi)}Q_{\des(\omega^{-1})}^{(r)}(X)\\
&= \sum_{(\D, \varphi) \in \park{m,n}}\sum_{(\D, \Upsilon) \in \std^{-1}_{r}(\D, \varphi)}q^{\area(\D, \Upsilon)}t^{\dinv(\D, \Upsilon)}x^{\we(\D, \Upsilon)} \\
&=\sum_{(\D, \Upsilon) \in \sspf{r}{m,n}}q^{\area(\D,\Upsilon)}t^{\dinv(\D, \Upsilon)}x^{\we(\D, \Upsilon)}\\
&= C^{(r)}_{(m,n)}(x_1, \dots, x_r;q,t).
\end{align*}
\end{proof}

\section{Affine Springer Fibers} \label{sec:springer}
Having defined and studied the higher rank Catalan polynomial, in this section we look at a geometric interpretation for the dinv statistic. This is inspired, on the one hand, by \cite{GMV} where the authors interpret the dinv of a parking function in terms of the geometry of affine Springer fibers on the affine flag variety. The idea to simultaneously look at several parabolic affine Springer fibers comes from \cite{GSV}, where a $\mathfrak{gl}$-version of the varieties we consider here appeared in connection to the representation theory of quantized Gieseker varieties.  Throughout this section, geometric considerations make it more natural to replace $\{1, \ldots, r\}$ by  $\{0, \ldots, r-1\}$ and $\{1, \ldots, m\}$  by $\{0, \ldots, m-1\}$.

\subsection{Affine partial flag varieties}\label{sec:affflags}
We consider the group $G := \SL_{m}$. We will first work over
the field $\K := \C((\ep))$ with ring of integers $\CO = \C[[\ep]]$. Let $\widetilde{P} \subseteq G(\C)$ be a parabolic subgroup, and $P \subseteq G(\CO) \subseteq G(\K)$ the pullback of $\tilde{P}$ under the evaluation $\ep \mapsto 0$. The group $P$ is known as a parahoric subgroup of $G(\K)$.

We will be interested in the following setting. Let $\bw:= (w_{0}, \dots, w_{r-1})$ be a composition of $m$ with $r$ parts (recall that our convention is that some of the parts may be zero), and $\widetilde{P} = \widetilde{P}^{\bw} \subseteq G$ the parabolic subgroup of block lower-triangular matrices

\begin{equation}
\widetilde{P}^{\bw} := \left(\begin{matrix} 
A_{00} & 0 & 0 & \cdots & 0 \\
A_{10} & A_{11} & 0  & \cdots & 0 \\
A_{20} & A_{21} & A_{22} & \cdots & 0 \\
\vdots & \vdots & \vdots & \ddots & \vdots \\
A_{r-1,0} & A_{r-1, 2} & A_{r-1, 3} & \cdots & A_{r-1, r-1}
\end{matrix}\right)
\end{equation}
where $A_{i,j}$ is a matrix of size $w_{i} \times w_{j}$. We let $P^{\bw} \subseteq G(\K)$ be the associated parahoric subgroup, with corresponding parahoric subalgebra $\mathfrak{p}^{\bw} \subseteq \mathfrak{g}(\K)$,
which is defined in a similar way. Note that, in fact, $\mathfrak{p}^{\bw} \subseteq \mathfrak{g}(\CO) \subseteq \mathfrak{g}(\K)$. Now the \emph{affine partial flag variety} is defined to be:
\begin{equation}
\Fl^{\bw} := G(\K)/P^{\bw}.
\end{equation}
This is an ind-variety.
For $g \in G(\K)$ we denote its coset in $\Fl^{\bw}$ by $[g]$. 
 There is a well-known moduli interpretation of $\Fl^{\bw}$ that we now recall, cf. \cite{LusztigSmelt}. Let $V := \C^{m}$, $V[[\ep]] := \CO \otimes_\C V$ and $V((\ep)) := \K \otimes_\C V$.
Recall that an $\CO$-lattice in $V((\ep))$ is a
free $\CO$-submodule $L \subseteq V((\ep))$ of  rank $m$.
 For example, $V[[\ep]]$ is such a lattice, known as the \emph{standard lattice}. Note that we automatically have that $\K \otimes_\C L = V((\ep))$ and that there exists $N > 0$ such that
\begin{equation}\label{eqn:bounds}
\ep^{N}V[[\ep]] \subseteq L \subseteq \ep^{-N}V[[\ep]].
\end{equation}

We say that a lattice $L \subseteq V((\ep))$ is of $\SL$-type if 
\[
\dim_\C(V[[\ep]]/L \cap V[[\ep]]) = \dim_\C(L/L\cap V[[\ep]]).
\]
Notice that due to \eqref{eqn:bounds} both sides are finite numbers. 
Then, $\Fl^{\bw}$ is the moduli space of flags of lattices
\begin{equation}\label{eqn: chain}
L_{0} \supseteq L_{1} \supseteq \cdots \supseteq L_{r} = \ep L_{0}
\end{equation}
such that $\dim_\C(L_{i}/L_{i+1}) = w_{i}$ and $L_{0}$ is of $\SL$-type. To see this isomorphism fix a basis $v_1, \dots, v_n$ of $V$. For each $i = 0, \dots, r$ let $V_{i} \subseteq V((\ep))$ be the free $\CO$-submodule generated by $v_{w_{0} + \cdots + w_{i-1} + 1}, \dots,$ $ v_{n}, \ep v_{1}, \dots \ep v_{n - w_{0} - \cdots - w_{i-1}}$. Note that $V_{0} = V[[\ep]]$, $V_{r} = \ep V[[\ep]]$ so we have a chain
\begin{equation}\label{eqn:standard flag}
V_{0} \supseteq V_{1} \supseteq \cdots \supseteq V_{r} = \ep V_{0}.
\end{equation}
Now the group $G(\K)$ acts transitively on the space of flags of lattices, and the stabilizer of the flag \eqref{eqn:standard flag} is precisely $P^{\bw}$, whence it follows that $\Fl^{\bw}$ is indeed the moduli space of flags of lattices. 

 We will extend the chain \eqref{eqn: chain} to have a lattice $L_{i}$ for $i \in \Z$ by setting $L_{i+r} := \ep L_{i}$.
We denote this chain $(L_i)$.

\subsection{Affine Springer fibers}\label{sec:affspringer} Fixing a basis $v_{1}, v_{2}, \dots, v_{m}$ of $V$, we will consider the following element of $\GL_{m}(\K)$: 
\begin{equation}
\widetilde{\gam} = \left(\begin{matrix}0 & 0 & \cdots &  0 & \ep \\ 1 & 0 & \cdots & 0 & 0 \\ 0 & 1 & \cdots & 0 & 0 \\ \vdots & \vdots & \ddots & \vdots & \vdots \\ 0 & 0 & \cdots & 1 & 0  \end{matrix}\right), \quad \widetilde{\gam}: v_{1} \mapsto v_{2} \mapsto \cdots \mapsto v_{m} \mapsto \ep v_{1} 
\end{equation} 
and $\gam := \widetilde{\gam}^{n}$. Since $m$ and $n$ are coprime, the element $\gam$ is \emph{nil-elliptic} in the sense of Kazhdan and Lusztig \cite{KazhdanLusztig}. In particular, the space

\begin{equation}
X^{\bw}_{(m,n)} := \{[g] \in \Fl^{\bw} \mid g^{-1}\widetilde{\gam}g \in \mathfrak{p}^{\bw}\} = \{(L_{i}) \in \Fl^{\bw} \mid \gam(L_{i}) \subseteq L_{i} \; \text{for every} \; i\}
\end{equation}
is a projective variety. Our goal is to study the homology of this space. In order to do this, first we re-interpret it in a standard way, see \cite{GMV,LusztigSmelt}. Recall the basis $v_{1}, \dots, v_{m}$ of $V$. We define $v_{i + m} := \ep v_{i}$, so that now we have $v_{i}$ for all $i \in \Z$, and the map $\gam$ can be succinctly written as $\gam: v_{i} \mapsto v_{i+n}$. 

Let us introduce an auxiliary variable $z := \ep^{1/m}$, so that we have an isomorphism of $\K$-vector spaces $V((\ep)) \to \C((z))$ given by $v_{i} \mapsto z^{i-1}$. Note that this isomorphism induces an isomorphism $V[[\ep]] \mapsto \C[[z]]$. The map $\gam$ then is simply multiplication by $z^{n}$, and multiplication by $\ep$ corresponds to multiplication by $z^{m}$. Thus, $X^{\bw}_{(m,n)}$ is also the moduli space of flags
\[
L_{0} \supseteq L_{1} \supseteq \cdots \supseteq L_{r} = z^{m}L_{0}
\]
where
\begin{itemize}
\item[(a)] $L_{i}$ is a $\C[[z^{m}, z^{n}]]$-submodule of $\C((z))$,
\item[(b)] $\dim_\C(L_{0}/\C[[z]]\cap L_{0}) = \dim_\C(\C[[z]]/\C[[z]]\cap L_{0})$, and
\item[(c)] $\dim_\C(L_{i}/L_{i+1}) = w_{i}$. 
\end{itemize}

Just as in the previous subsection, we define $L_{i+r} := z^{m}L_{r}$. Thus, $L_{i}$ is well-defined for every $i \in \Z$.

Let us remark that, when $w_{i} \leq 1$ for every $i$, then $\Fl^{\bw}$ is the affine flag variety. In this case, $X^{\bw}_{(m,n)}$ has been extensively studied in  \cite{GMV, hikita, KazhdanLusztig, LusztigSmelt} and many more. When $w_{i} = m$ for some $i$ (and thus necessarily $w_{j} = 0$ for $j \neq i$) the variety $\Fl^{\bw}$ is the affine Grassmannian and $X^{\bw}_{(m,n)}$ has been studied in, for example, \cite{GorskyMazin1, GorskyMazin2, hikita}. For general $\bw$, the varieties $X^{\bw}_{(m,n)}$ have appeared in \cite{hikita}, where it was proven that they admit an affine paving with cells indexed by $\sspf{\bw}{m,n}$. We will re-derive Hikita's results below and give a formula for the dimension of the affine cells in terms of affine permutations, in the spirit of \cite{GMV}. 

\subsection{Loop rotation and fixed points}\label{sect: fixed points} The field $\C((z))$ has a natural valuation 
\begin{equation}
\nu: \C((z))^{\times} \to \Z, \qquad g = \sum_{i}g_{i}z^{i} \mapsto \min\{i \mid g_{i} \neq 0\}.
\end{equation}
For $M \subseteq \C((z))$, we set $\nu(M) := \{\nu(g) \mid g \in M\setminus\{0\}\} \subseteq \Z$.

We have a $\C^{\times}$-action on $\C((z))$ given by $s\cdot g(z) = g(sz)$. Note that, upon the identification $\C((z)) = \C((\varepsilon^{1/m}))$, this is a positive rational multiple of the loop rotation, so it induces an action on $\Fl^{\bw}$ for every composition $\bw = (w_{0}, \dots, w_{r-1}) \models m$. Since the map $\gam$ is multiplication by $z^{n}$, which is homogeneous under the $\C^{\times}$-action, the variety $X^{\bw}_{(m,n)}$ is closed under the $\C^{\times}$-action on $\Fl^{\bw}$. We would like to give a description of the fixed points of this action. The chain
\[
L_{0} \supseteq L_{1} \supseteq \cdots \supseteq L_{r} = z^{m}L_{0}
\]
is fixed under the $\C^{\times}$-action if and only if each $\C[[z^{m}, z^{n}]]$-submodule $L_{i}$ is. It is known (see e.g. \cite{Yun}) that such a submodule $L_{i}$ is fixed by the $\C^{\times}$-action if and only if it is generated by Laurent monomials in $z$.
If we consider $\nu(L_{i}) \subseteq \Z$ we have a chain of subsets of $\Z$, 
\[
\nu(L_0) \supseteq \nu(L_1) \supseteq \cdots \supseteq \nu(L_{r}) = m + \nu(L_0).
\]
The fact that $L_{i}$ is closed under multiplication by both $z^m$ and $z^n$ implies that $\nu(L_{i})$ is closed under adding $m$ and $n$. Following \cite{GorskyMazin1}, we call such sets $(m, n)$-invariant subsets of $\Z$. Thus, we see that the $\C^{\times}$-fixed points of $X^{\bw}_{(m,n)}$ correspond to chains,
\[
N_{0} \supseteq N_{1} \supseteq \cdots \supseteq N_{r} = N_{0} + m
\]
of $(m, n)$-invariant subsets of $\Z$, subject to the conditions,
\begin{itemize}
    \item[(a)] $|N_{i} \setminus N_{i+1}| = w_{i}$, and 
\label{item-a}
    \item[(b)] $|\Z_{\geq 0} \setminus N_{0}| - |N_{0} \setminus \Z_{\geq 0}| = 0$.
\label{item-b}
\end{itemize}

Note that we can extend this chain to a chain of $(m, n)$-invariant subsets indexed by $i \in \Z$ by declaring $N_{i + r} = N_{i} + m$. Note that $N_{i} \supseteq N_{i+1}$ for every $i \in \Z$. 

For every chain $(N_{i})$ we would like to associate an $(m,r)$-affine composition. To do this, consider the following slight modification of Definition \ref{def:affcomp}. A function $f: \Z \to \Z$ is a \emph{0-centric $(m,r)$-affine composition} if the following holds:
\begin{enumerate}
\item $f(x+m) = f(x) + r$ for every $x \in \Z$,
\label{item- m period}
\item $f^{-1}[0,r-1]$ contains exactly $m$ elements, which are pairwise distinct modulo $m$,
\label{item- distinct mod m}
\item $\sum\limits_{x \in f^{-1}[0,r-1]}x \quad  = 0 + \cdots + m-1 = \binom{m}{2}$. 
\label{item- sum normalized}
\end{enumerate}
Let us denote by $\affcz{}(m,r)$ the set of $0$-centric affine compositions. It is clear that we have a natural bijection $\affc{}(m,r) \to \affcz{}(m,r)$ inducing bijections $\affc{\bw}^{n}(m,r) \to \affcz{\bw}^{n}(m,r)$ for every composition of $m$ with $r$ parts and $n > 0$, for the obvious definition of $\affcz{\bw}^{n}(m,r)$. The $0$-centric affine compositions are, however, more natural from the point of view of chains of $(m,n)$-invariant subsets of $\Z$, as the following lemma shows. 

\begin{lemma}\label{lem:from chains to comps}
Let $(N_{i})$ be a chain of $(m,n)$-invariant subsets of $\Z$ satisfying (a) and (b) above. Define $f: \Z \to \Z$ via
\[
f^{-1}(i) = N_{i}\setminus N_{i+1}
\]
Then, $f$ is an $n$-stable $0$-centric $(m,r)$-affine composition of weight $\bw$. 
\end{lemma}
\begin{proof}
First note that every integer $k$ belongs to an $N_{i}$ for some $i$, and that $(N_{i} \setminus N_{i+1}) \cap (N_{j} \setminus N_{j+1}) = \emptyset$ if $i \neq j$.
 This says that $f: \Z \to \Z$ is indeed well-defined. Now assume $f(x) = i$, so $x \in N_{i} \setminus N_{i+1}$. Since $N_{i+r} = N_{i} + m$,  it follows that $x+m \in N_{i+r} \setminus N_{i+r+1}$, so $f(x+m) = i+r$, as needed. 

Now, by definition, $f^{-1}[0,r-1] = \bigcup_{i= 0}^{r-1} f^{-1}(i) = (N_{0} \setminus N_{1})\sqcup (N_{1} \setminus N_{2}) \sqcup \cdots \sqcup (N_{r-1} \setminus N_{r}) = N_{0} \setminus N_{r}$. Since $|N_{i} \setminus N_{i+1}| = w_{i}$, and $w_{0} + \cdots + w_{r-1} = m$, it follows that $f^{-1}[0,r-1]$ has exactly $m$ elements, and we need to verify that their residues mod $m$ are pairwise distinct. Let $i < j \in \Z$ have the same residue modulo $m$, say $j = i + bm, b >0$. If $i \in N_0$ then, since $N_{0}$ is stable under addition of $m$,
$i + (b-1)m \in N_{r} = N_{0} + m$ and so $j = i+bm \in N_{r}$. Hence we cannot have both $i,j \in N_{0} \setminus N_r $ and the result follows. 

Due to the normalization condition \eqref{item-b}
of $N_0$, the sum of elements in $N_0 \setminus (m+N_0)$ does not depend on $N_0$, and it is easy to see that for $N_0 = \Z_{\geq 0}$ this is $\binom{m}{2}$. We conclude that, indeed, $f \in \affcz{}{(m, r)}$. It is obvious that the weight of $f$ is precisely $\bw$. Finally, if $f(x) = i$ then $x \in N_{i}$, so $x + n \in N_{i}$ as well, from which it follows that $f(x+n) \geq i$. Thus, $f$ is $n$-stable.
\end{proof}

\subsection{Geometric interpretation of standardization}\label{sec:geo standardization} In this section, we provide a geometric perspective of the standardization procedure in Section \ref{sec:standardization}. In order to do this, we state a few remarks first. By Lemma \ref{lem:from chains to comps} we get a map $X^{\bw}_{(m,n)} \to \affcz{\bw}^{n}(m,r)$. As a consequence of the results in Section \ref{sect: fixed points}, the fibers of this map give a paving on $X^{\bw}_{(m,n)}$. For an element $f \in \affc{\bw}^n(m,r)$, let $C_{f} \subseteq X^{\bw}_{(m,n)}$ be its corresponding cell. That is,
\begin{equation}
C_{f} := \{(L_{i})_{i \in \Z} \mid \nu(L_{i}) \setminus \nu(L_{i+1}) = f^{-1}(i)\}.
\end{equation}

Now let $\overline{\bw} := (1, 1, \dots, 1)  = (1^m)$. Note that, for any other composition $\bw$ of $m$ we have a projection:
\begin{equation}
\pi_{\bw}: X^{\overline{\bw}}_{(m,n)} \to X^{\bw}_{(m,n)}
\end{equation}
given by simply forgetting some of the subspaces in the chain. Note that the affine cells in $X^{\overline{\bw}}_{(m,n)}$ are indexed by $n$-stable affine \emph{permutations} or, equivalently, by parking functions. Recall from Section \ref{sec:standardization} that, given a parking function $(\D, \varphi)$ and a weight $\bw$, we have defined the $\bw$-semistandardization $\sstd{\bw}(\D, \varphi)$ of $(\D, \varphi)$. The following proposition is then immediate. 

\begin{proposition}
For a composition $\bw$ of $m$ and an affine composition $f \in \affcz{\bw}(m,r)^{n}$, let $C^{\bw}_{f}$ denote the corresponding affine cell in $X^{\bw}_{(m,n)}$. Then,
\[
\pi_{\bw}^{-1}(C^{\bw}_{f}) = \bigsqcup_{\sstd{\bw}(\sigma) = f}C^{\overline{\bw}}_{\sigma}.
\]
\end{proposition}

Our goal is to geometrically distinguish the standardization $\std(\D, \Upsilon)$ from those parking functions that $\bw$-semistandardize to $(\D, \Upsilon)$. In order to achieve this, we will use the following lemmas. Throughout, for $(L_{i}) \in X^{\bw}_{(m,n)}$ let $f: \Z \to \Z$ be defined by $f^{-1}(i) = \nu(L_{i})\setminus \nu(L_{i+1})$. \footnote{Note that, technically, $\nu$ is not defined at $0$, so the correct statement ought to be $\nu(L_{i} \setminus \{0\}) \setminus \nu(L_{i+1} \setminus \{0\})$. For ease of notation, we choose to abuse the notation and simply write $\nu(L_{i}) \setminus \nu(L_{i+1})$.}
\begin{lemma}\label{lem:nu}
For each $i$, $\nu(L_{i}) = f^{-1}(\Z_{\geq i})$.
\end{lemma}
\begin{proof}
First, we claim that $\bigcap_{i} L_{i} = 0$. Indeed, assume $g \in \bigcap L_{i}$. Then, $g \in \bigcap_{i} L_{ri} = \bigcap_{i} z^{mi}L_{0}$. Recall that there exists $N > 0$ such that $L_{0} \subseteq z^{-N}\C[[z]]$. So then $\bigcap_{i}z^{mi}L_{0} \subseteq \bigcap_{i} z^{-N + mi}\C[[z]]$ and the latter intersection is clearly zero.
This implies that $\bigcap_{j \geq i} L_{j} = 0$. Now, $\nu(L_{i}) \supseteq \nu(L_{i+1}) \supseteq \cdots$ and, by the statement above, $\bigcap_{j \geq i}\nu(L_{j}) = \emptyset$. Thus, $\nu(L_{i}) = \bigsqcup_{j \geq i} \nu(L_{j}) \setminus \nu(L_{j+1}) = \bigsqcup_{j \geq i} f^{-1}(j) = f^{-1}(\Z_{\geq i})$. 
\end{proof}

\begin{lemma}
Suppose that $\nu(L_{i})$ contains $\{j, j+1, j+2, \dots\}$ for some $j \in \Z$. Then, $L_{i}$ contains $z^{j}\C[[z]]$. 
\end{lemma}
\begin{proof}
First, note that since $L_{i}$ contains $L_{rk}$ for some $k$, $L_{i}$ contains $z^{mk}L_{0}$. By our assumptions on $L_{0}$, it follows that $L_{i}$ contains $z^{N}\C[[z]]$ for some $N \gg 0$. If $N \leq j$ there is nothing to do, so we may assume that $N > j$. Then $N -1 \in \{j, j+1, \dots\}$ and so there exists $g = z^{N-1} + \alpha \in L_{i}$, where $\alpha \in  z^{N}\C[[z]]$. Since $z^{N}\C[[z]] \subseteq L_{i}$ it follows that $z^{N-1} \in L_{i}$ and thus $z^{N-1}\C[[z]] \subseteq L_{i}$. The result follows by induction on $N-j$. 
\end{proof}

\begin{corollary}
Assume $\nu(L_{i}) = \{j, j+1, j+2, \dots\}$ for some $j \in \Z$. Then $L_{i} = z^{j}\C[[z]]$. 
\end{corollary}

We can now state the main theorem of this section. 

\begin{theorem}\label{thm:geo std}
Let $\sigma$ be such that $\sstd{\bw}(\sigma) = f$. Then, the map
\[
\pi_{\bw}: C_{\sigma}^{\overline{\bw}} \to C_{f}^{\bw}
\]
is injective if and only if $\sigma = \std(f)$. 
\end{theorem}

\begin{proof}
Let us, first, set up some conventions for the proof. We extend the weight vector $\bw = (w_0, \dots, w_{r-1}) \in \Z_{\geq 0}^{r}$ to a vector $(w_{i})_{i \in \Z}$ by setting $w_{i} = w_{i+r}$. Note that we get, for every $i \in \Z$, $|f^{-1}(i)| = w_{i}$. For each $i$, let us assume without loss of generality that the set $f^{-1}(i)$ is ordered as $\{\ell_{i, 1} < \cdots < \ell_{i, w_{i}}\}$, and let $\sigma = \std(f)$. Note that $\sigma$ is given as follows:
\[
\begin{array}{rcl}
f^{-1}(0) &=& \{\sigma^{-1}(0) < \sigma^{-1}(1) < \dots < \sigma^{-1}(w_0-1)\}, \\
f^{-1}(1) &= &\{\sigma^{-1}(w_0) < \sigma^{-1}(w_0+1) < \dots < \sigma^{-1}(w_0+w_1-1)\},\\
 &\vdots &  \\
f^{-1}(r-1) & = & \{\sigma^{-1}(w_0+\cdots + w_{r-2}) < \dots < \sigma^{-1}(w_0+\cdots + w_{r-1}-1=m-1)\}.
\end{array}
\]

Let us fix an element $(L_{i}) \in C_{f}^{\bw}$. We will use the notation of the proof of Theorem \ref{thm:geo dinv}, in particular, for $i \in \Z$ such that $f^{-1}(i) = \{\ell_{i,1} < \dots < \ell_{i, w_{i}}\}$ we have unique elements $g_{i,j} \in L_{i}$, $j = 1, \dots, w_{i}$ such that
\[
g_{i,j} = z^{\ell_{i,j}} + \sum_{\substack{p > \ell_{i,j} \\ f(p) < i}}\lambda^{p}_{\ell_{i,j}}z^{p}.
\]

We claim that $L_{i}$ has a topological basis given by
$\{g_{k,j} \mid k \geq i, 1 \leq j \leq w_{k}\}$.
This means that every element $g \in L_{i}$ can be uniquely written as a (possibly infinite) sum
\[
g = \sum_{\substack{k \geq i \\ 1 \leq j \leq w_{k}}}\alpha_{k,j}g_{k,j}, \qquad \alpha_{k,j} \in \C.
\]
Indeed, since $\nu(g) \in \nu(L_{i})$ it follows from Lemma \ref{lem:nu} that there exists $i \leq j$, $j \in \{1, \dots, w_{k}\}$ and a unique element $\alpha_{k,j} \in \C$ such that $\nu(g - \alpha_{k,j}g_{k,j}) > \nu(g)$. Since $g - \alpha_{k,j}g_{k,j} \in L_{i}$, the result now follows by induction after observing that since $L_{i}$ is a lattice, its $z$-adic topology is complete and separated. To conclude, we have
\[
L_{i} = \spann\{g_{i, 1}, \dots, g_{i, w_{i}}, g_{i+1, 1}, \dots, g_{i+1, w_{i+1}}, \dots\},
\]
where, for the remainder of the proof, ``span" refers to topological span. That is, we allow infinite sums that converge in the $z$-adic topology. 

Now, we want to show that there exists at most one element $(M_{t}) \in C^{\overline{\bw}}_{\sigma}$ such that $\pi(M_{t}) = (L_{i})$. The equation $\pi(M_{t}) = (L_{i})$ forces,
\[
\begin{array}{rcl}
M_{0} & = & L_{0} = \spann\{g_{0,1}, \dots, g_{0, w_{0}}, g_{1,1}, \dots, g_{1, w_{1}}, \dots\} \\
M_{w_{0} - 1} & = & L_{1} = \spann\{g_{1,1}, \dots, g_{1,w_{1}}, g_{2,1}, \dots, g_{2,w_{2}}, \dots\}\\
 & \vdots & \\ 
M_{m-1} & = & L_{r-1} = \spann\{g_{r-1, 1}, \dots, g_{r-1, w_{r-1}}, g_{r,1}, \dots, g_{r, w_{r}}, \dots\}.
\end{array}
\]
In particular, it is necessary to prove that there exists a unique way to choose subspaces $M_{1}, M_{2}, \dots$ such that $(M_{t}) \in C^{\overline{\bw}}_{\sigma}$
for $\sigma$ the standardization of $f$.

It suffices to prove that there is a unique choice for $M_{1}$ since the others will follow similarly. Indeed, by Lemma \ref{lem:nu} we must have $\nu(M_{1}) = \sigma^{-1}(\Z_{\geq 1}) = \{\ell_{0,2}, \dots, \ell_{0,w_{0}}, \ell_{1,1}, \dots, \ell_{1, w_{1}}, \dots\}$. We claim that the only choice is to have:
\begin{equation}\label{eq:m1 forced}
M_{1} = \spann\{g_{0,2}, \dots, g_{0, w_{0}}, g_{1,1}, \dots, g_{1,w_{1}}, \dots\}.
\end{equation}
Indeed, since $M_{1} \subseteq M_{0} = L_{0}$, every element of $M_{1}$ has the form $g = \sum_{0 \leq k, 1 \leq j \leq w_{k}}\alpha_{k,j}g_{k,j}$. Assume that there exists $g \in M_{1}$ with $\alpha_{0,1} \neq 0$. Since $M_{w_{0}-1} = L_{1} \subseteq M_{1}$ and $L_{1} = \spann\{g_{1,1}, g_{1,2}, \dots\}$ then $g' = \sum_{j = 1}^{w_{0}}\alpha_{0,j}g_{0,j} \in M_{1}$.  Moreover, because $\alpha_{0,1} \neq 0$, we obtain $\nu(g') = \nu(g_{0,1}) = \sigma^{-1}(0) \not\in \nu(M_1)$, which is a contradiction.  Thus, $M_1 \subseteq \spann\{g_{0,2}, \dots, g_{0,w_{0}}, \dots\}$ and, since $\dim_{\C}(M_{0}/M_{1}) = 1$, we have that $M_1$ is forced to be as in \eqref{eq:m1 forced}. Using a similar strategy, we see that $M_{2}$, $M_{3}, \dots$ are also uniquely defined. Thus, $\pi_{\bw}: C^{\overline{\bw}}_{\sigma} \to C^{\bw}_{f}$ is injective.

Now let $\tilde{\sigma}$ be such that $\sstd{\bw}(\tilde{\sigma}) = f$ but $\tilde{\sigma} \neq \std(f)$. This means that 
\[
\begin{array}{rcl}
f^{-1}(0) &=& \{\tilde{\sigma}^{-1}(0), \tilde{\sigma}^{-1}(1), \dots, \tilde{\sigma}^{-1}(w_0-1)\} \\
f^{-1}(1) &= &\{\tilde{\sigma}^{-1}(w_0), \tilde{\sigma}^{-1}(w_0+1), \dots, \tilde{\sigma}^{-1}(w_0+w_1-1)\}\\
 &\vdots &  \\
f^{-1}(r-1) & = & \{\tilde{\sigma}^{-1}(w_0+\cdots + w_{r-2}), \dots,  \tilde{\sigma}^{-1}(w_0+\cdots + w_{r-1}-1=m-1)\}
\end{array}
\]
but there exist $h \in \{0, \dots, r-1\}$ and $i, j \in \{w_0 + \cdots + w_{h}, \dots, w_{0} + \cdots + w_{h} + w_{h+1}-1\}$ with $i < j$ and $\tilde{\sigma}^{-1}(j) < \tilde{\sigma}^{-1}(i)$. We may and will assume that $|j-i|$ is minimal with this property. Note that $\tilde{\sigma}^{-1}(i) - \tilde{\sigma}^{-1}(j) \neq n$, as $\tilde{\sigma}$ does not have inversions of height $n$.

Let us take the element $(\bar{L}_{i}) \in C^{\bw}_{f}$ defined by
\[
\bar{L}_{i} = \spann\{z^{\ell_{k,j}} \mid k \geq i, 1 \leq j \leq w_{k}\}.
\]
In other words, $(\bar{L}_{i})$ is the unique element of $C^{\bw}_{f}$ that is fixed by loop rotation. We will find a collection of distinct elements $(M^{\alpha}_{t}) \in C^{\overline{\bw}}_{\tilde{\sigma}}$ such that $\pi_{\bw}(M^{\alpha}_{t}) = (\bar{L}_{i})$ for every $\alpha$, which will finish the proof of the theorem. It suffices to define $M_{0}, \dots, M_{m-1}$, since all the other spaces are fixed by the condition $M_{i+m} = z^{m}M_{i}$. As in the previous case, for any $\alpha$ one obtains
\[
M^{\alpha}_{0} = \bar{L}_{0}, \quad M^{\alpha}_{w_{0} - 1} = \bar{L}_{1}, \quad \dots,\quad  M^{\alpha}_{m-1} = \bar{L}_{r-1}.
\]
Now, consider $M^{\alpha}_{t}$ given by
\[
M^{\alpha}_{t} :=
\begin{cases}
\spann\{z^{\tilde{\sigma}^{-1}(s)} \mid s \geq t\}, & t \not\in [i+1, j];\\
\spann\{z^{\tilde{\sigma}^{-1}(t)}, z^{\tilde{\sigma}^{-1}(t+1)}, \dots, z^{\tilde{\sigma}^{-1}(j-1)}, z^{\tilde{\sigma}^{-1}(j)} + \alpha z^{\tilde{\sigma}^{-1}(i)}, z^{\tilde{\sigma}^{-1}(j+1)}, \dots\}, &t \in [i+1, j].
\end{cases}
\]
 It is clear that $M^{\alpha}_{t}$ is closed under multiplication by $z^m$. To check that it is closed under multiplication by $z^n$, note that we only need to check this for $t \in [i+1, n]$.  If $k \in \{i+1, \dots, j-1\}$ then we cannot have $z^{n}z^{\tilde{\sigma}^{-1}(k)} = z^{\tilde{\sigma}^{-1}(j)}$ since this would imply that $\tilde{\sigma}^{-1}(k) < \tilde{\sigma}^{-1}(j) < \tilde{\sigma}^{-1}(i)$ and $|k-i| < |j-i|$, a contradiction. Similarly, we cannot have $z^{n}z^{\tilde{\sigma}^{-1}(i)} = z^{\tilde{\sigma}^{-1}(k)}$ for $k \in \{i+1, \dots, j-1\}$ since this implies that $\tilde{\sigma}^{-1}(j) < \tilde{\sigma}^{-1}(i) < \tilde{\sigma}^{-1}(k)$ and $|k-j| < |i-j|$, again a contradiction. Together, these two conditions imply that $M^{\alpha}_{t}$ is closed under multiplication by $z^{n}$. So we may find an entire affine line of elements in $C^{\overline{\bw}}_{\tilde{\sigma}}$ mapping to the same element of $C^{\bw}_{f}$, as desired.
\end{proof}

\begin{corollary}\label{cor:bound dim}
For $f \in \affc{\bw}^n(m, r)$,
\[
\dim_{\C}(C_f) \geq \codinv(f).
\]
\end{corollary}
\begin{proof}
From Theorem \ref{thm:geo std} we get that $\dim_{\C}(C_f) \geq \dim_{\C}(C_{\std(f)})$. By \cite{GMV, hikita}, $\dim_{\C}(C_{\std(f)}) = \codinv(\std(f))$ and, by definition, $\codinv(\std(f)) = \codinv(f)$. 
\end{proof}

\subsection{Dimensions of cells}\label{sec:dimension cells} In this section we prove that the cells $C_{f}$ are affine spaces and compute their dimensions. Namely, we have the following result. 

\begin{theorem}\label{thm:geo dinv}
Let $f \in \affc{\bw}^n(m,r)$ and let $C_{f} \subseteq X^{\bw}_{(m,n)}$ be the cell corresponding to $f$. That is, 
$
C_f = \{(L_{i})_{i \in \Z} \mid \nu(L_{i}) \setminus \nu(L_{i+1}) = f^{-1}(i)\}.
$
Then, $C_f$ is an affine space with dimension given by 
\[
\dim_\C(C_{f})
= \codinv(f).
\]
\end{theorem}

\begin{proof}
Extend $\bw$ to a sequence $(w_{i})_{i \in \Z}$ by setting $w_{i+r} = w_{i}$ for every $i$. In particular, if $f \in \affc{\bw}^n(m,r)$ then $|f^{-1}(i)| = w_{i}$ for every $i \in \Z$. Recall that $C_{f}$ is the set of chains $L_{0} \supseteq \cdots \supseteq L_{r} = z^{m}L_{0}$ of $\C[[z^{m}, z^{n}]]$-submodules of $\C((z))$, where $\nu(L_{i}) \setminus \nu(L_{i+1}) = f^{-1}(i)$.

Let us take $i \in \Z$ such that $w_{i} \neq 0$, and write  $f^{-1}(i) =
\{\ell_{i, 1} < \ell_{i, 2} < \cdots < \ell_{i, w_{i}}\}$.
Pick $(L_\bullet) \in C_{f}$. 
 Thus, for
each $j = 1, \dots, w_{i}$ we can find $g_{i,j} \in L_{i}$ such that
\[
g_{i, j} = z^{\ell_{i,j}} + \sum_{p > \ell_{i,j}}\lambda^{p}_{\ell_{i,j}}z^{p}
\]
and $\lambda^{p}_{\ell_{i,j}} \in \C$. Note that we may assume that $\lambda^{p}_{\ell_{i,j}} = 0$ if $f(p) \geq f(\ell_{i,j}) = i$. We claim that, under these conditions, the functions $g_{i,j}$ are unique. Indeed, if this is not the case then we would have $\dim_\C(L_{i}/L_{i+1}) > w_{i}$, a contradiction. Thus, we obtain, for each $\alpha \in \Z$, a function $g^{\alpha} \in \C((z))$ such that
\[
g^{\alpha} = z^{\alpha} + \sum_{\substack{\beta > \alpha \\ f(\beta) < f(\alpha)}} \lambda^{\beta}_{\alpha}z^{\beta}.
\]
Since $L_{i+r} = z^{m}L_{i}$ for every $i$, we have that $g^{\alpha + m} = z^{m}g^{\alpha}$. Thus, it is enough to consider $\alpha = 0, \dots, m-1$, from which we see that the dimension
of the cell is at most the number of inversions of $f$. 
Below we  show that 
the parameters coming from inversions of
height $>n$ can all be re-expressed in terms of those of smaller height,
showing that $\dim_{\C}(C_f) \le \codinv(f)$. The result will then follow from this combined with Corollary \ref{cor:bound dim}.

Now, since $z^{n}L_{f(\alpha)} \subseteq L_{f(\alpha)}$ and $f(\alpha + n) \geq f(\alpha)$ we have that $z^{n}g^{\alpha} - g^{\alpha + n} \in L_{f(\alpha)}$. Thus,
\begin{equation}\label{eqn: difference}
\sum_{\beta > \alpha} \lambda^{\beta}_{\alpha}z^{\beta + n} - \sum_{\beta > \alpha + n}\lambda^{\beta}_{\alpha + n}z^{\beta} = \sum_{\beta > \alpha + n}(\lambda^{\beta-n}_{\alpha} - \lambda^{\beta}_{\alpha + n})z^{\beta} \in L_{f(\alpha)}.
\end{equation}
Let $\beta > \alpha + n$. If $f(\beta) \geq f(\alpha)$, then we can eliminate the $z^{\beta}$ term in \eqref{eqn: difference} and we still get an element of $L_{f(\alpha)}$.

If $f(\beta) < f(\alpha)$, then $f(\beta - n) \leq f(\beta) < f(\alpha)$ and $f(\beta) < f(\alpha) \leq f(\alpha+n)$,
so none of the terms $\lambda^{\beta - n}_{\alpha}$ nor $\lambda^{\beta}_{\alpha + n}$ are guaranteed to vanish. However, it will be shown that $\lambda^{\beta - n}_{\alpha} - \lambda_{\alpha + n}^{\beta}$ can be expressed in terms of other coefficients. 

Indeed, since $f(\beta) < f(\alpha)$, there does not exist $h \in L_{f(\alpha)}$ such that $\nu(h) = \beta$. We claim that we can eliminate terms of smaller degree in \eqref{eqn: difference} using elements of $L_{f(\alpha)}$. Indeed, let us order the set:
\[\{\beta \mid \beta > \alpha+n \; \text{and} \; f(\beta) < f(\alpha)\} = \{\beta_0 < \beta_1 < \cdots\}. \]
Note that, if $\alpha + n < \beta < \beta_0$, then $f(\beta) \geq f(\alpha)$, so $g^{\beta} \in L_{f(\beta)} \subseteq L_{f(\alpha)}$ and we can eliminate the term $z^{\beta}$ in \eqref{eqn: difference} while still staying in $L_{f(\alpha)}$. After eliminating terms of lower degree in \eqref{eqn: difference} using elements of $L_{f(\alpha)}$,
 the coefficient of $z^{\beta_{0}}$ has to vanish. Hence, we have an expression of the form
\[
\sum_{\beta > \beta_{0}}(\lambda^{\beta - n}_{\alpha} - \lambda^{\beta}_{\alpha+n} - \tilde{\lambda}^{\beta}_{\alpha})z^{\beta}
\] 
where $\tilde{\lambda}^{\beta}_{\alpha}$ is some expression involving $\lambda^{i}_{j}$ where $i-j < n$. We can now get rid of all terms smaller than $\beta_1$ and so on. Thus, $\lambda^{\beta-n}_{\alpha}$ can be expressed in terms of $\lambda^{\beta}_{\alpha + n}$ and higher order monomials on $\lambda^{a}_{b}$ with $a - b < n$. 

Therefore, the dimension of $C_{f}$ is, at most, the number of pairs $(\alpha, \beta) \in \INV(f)$ such that $(\alpha, \beta + n) \not\in \INV(f)$. These correspond to inversions of $f$ of height less than $n$. Indeed, to $(\alpha, \beta)$ we can associate $(\alpha, \beta - kn)$, where $k$ is maximal so that $\beta - kn > \alpha$.
\end{proof}

\begin{corollary}
The Poincar\'e polynomial of $X^{\bw}_{(m,n)}$ is given by
\[
\sum_{\sigma \in \minlength{\bw}{m}{n}} t^{2\codinv(\sigma)}.
\]
\end{corollary}

\section{DAHA and a finite shuffle theorem}\label{sec:DAHA}

Theorem \ref{thm:main1} allows us to connect the higher rank Catalan polynomials $C^{(r)}_{(m,n)}(x_1,\dots, x_r;q,t)$ to Cherednik's double affine Hecke algebras (DAHA).  Indeed,  the Hikita polynomial $\hik_{(m,n)}(X;q,t) \in \Sym$ arises under the action of the the elliptic Hall algebra in celebrated Rational Shuffle Theorem \cite{CM,  GorskyNegut,  Mellit-rational}.  Moreover, both  the elliptic Hall algebra and $\Sym$ arise as inverse limits of the spherical subalgebra of the DAHA and of the ring of symmetric polynomials, respectively.  In this section, we combine these results and make the relationship between the DAHA and the higher rank Catalan polynomials precise by giving a finite dimensional analogue of the Rational Shuffle Theorem.

\subsection{The Elliptic Hall Algebra} 

The \emph{elliptic Hall algebra} $\mathcal{E}$ was originally introduced by Burban and Schiffmann as the Hall algebra of the category of coherent sheaves on an elliptic curve \cite{BurbanSchiffmann}. Since its inception, this algebra has had various algebraic, geometric, combinatorial, and even topological incarnations \cite{BGHT-identities, MortonSamuelson, negut, SchiffmannVasserotK}. For current purposes, it suffices to define $\mathcal{E}$ as the $\C(q,t)$-algebra generated by the infinite family $\lbrace P_{(m,n)} \rbrace$ with $(m,n) \in \Z^2 \setminus \lbrace(0,0)\rbrace$ modulo some relations. Since we do not need these relations, we omit the details, but refer the reader to \cite{BurbanSchiffmann, SchiffmannVasserotK} for a complete description of this algebra. 

In this article, we will consider only the (positive part of the) \emph{elliptic Hall algebra} $\EHA$ consisting of $P_{(m,n)}$ with $(m,n) \in \Z_{\geq 0}^2 \setminus \lbrace(0,0)\rbrace$ . It is known, see e.g. \cite[Section 4.4]{Bergeron} that $\EHA$ is generated by the elements $P_{(0,1)}$ and $P_{(1,0)}$.

In \cite{FeiginTsymbaliuk} and \cite{SchiffmannVasserotK}, Feigin and Tsymbaliuk and, independently, Schiffmann and Vasserot construct a \emph{geometric action} of $\EHA$ on the $(\C^{\times})^2$-equivariant $K$-theory of the Hilbert scheme of points on $\C^2$. This action was later made explicit by
Negu\c{t} in \cite{negut}. According to \cite[Theorem 4.2]{GorskyNegut}, this $K$-theory is isomorphic to $\Sym$, with the $K$-theory classes of fixed points corresponding to the modified Macdonald polynomials $\tilde{H}_{\lambda}(X;q,t)$, where $\lambda$ is a partition. In particular, this implies that the $\tilde{H}_{\lambda}(X;q,t)$ arise as common eigenfunctions of certain $\EHA$ operators.  Denote this geometric representation by 
\begin{equation}
 \Geom: \EHA \to \End(\Sym).
\end{equation}

It turns out this geometric action is intimately tied with the Hikita polynomial $\hik_{(m,n)}(X;q,t)$ defined in \eqref{eq:Hikita}. The following result is the celebrated \emph{Rational Shuffle Theorem}, conjectured by Gorsky and
Negu\c{t}
in \cite{GorskyNegut} and proven in this generality by Mellit in \cite{Mellit-rational}, see also \cite{CM}.

\begin{theorem}[Rational Shuffle Theorem]\label{thm:shuffle}
Consider  $1 \in \Sym$. Then, for coprime $m, n \in \Z_{\geq 0}^{2}$:
\[
\Geom(P_{(m,n)})\cdot 1 = \hik_{(m,n)}(X;q,t).
\]
\end{theorem}

\subsection{Spherical DAHA} \label{subsec:SphDaha} One of the incarnations of the elliptic Hall algebra $\EHA$ is that it appears as an inverse limit of spherical double affine Hecke algebras \cite{SchiffmannVasserotK},
\begin{equation}\label{eq:eha limit}
\EHA \cong \varprojlim_{r}\DAHA{r}.
\end{equation}
Let us recall that $\DAHA{r}$ is the subalgebra of the spherical subalgebra of the type $\mathfrak{gl}(r)$ double affine Hecke algebra (DAHA), generated by elements $P^{(r)}_{(m,n)}$ for $(m, n) \in \Z^{2}_{\geq 0}\setminus \lbrace(0,0)\rbrace$.
As with the elliptic Hall algebra, we omit the details of the relations and refer the reader to \cite{SchiffmannVasserotMac}. 

The algebras $\DAHA{r}$ admit natural surjections $\DAHA{r} \to \DAHA{r-1}$ sending $P^{(r)}_{(m,n)} \mapsto P^{(r-1)}_{(m,n)}$ that give rise to the algebra morphisms $\pi_r: \EHA \twoheadrightarrow \DAHA{r}$ in \eqref{eq:eha limit}, under which 
\[
P_{(m,n)} = \varprojlim_{r} P_{(m,n)}^{(r)}.
\]

In addition, each $\DAHA{r}$ comes equipped with a faithful \emph{polynomial representation} on the algebra of symmetric polynomials $\Symr$, denoted by 
\begin{equation} \pol{r}: \DAHA{r} \to \End(\Symr). \end{equation} 

It is well known that $\Sym$ arises as an inverse limit of $\Symr$. Hence, there exist morphisms that can be defined on power sums $p_k$ and extended linearly,
\begin{align}
\xi_r: \Lambda \to \Lambda_r & \qquad ; \qquad  p_k(X) \mapsto p_k(x_1,\dots,x_r) \label{eq:Symlimit} \\
\overline{\xi_r}: \Lambda_r \to \Lambda & \qquad ; \qquad p_k(x_1,\dots, x_r) \mapsto p_k(X)  \;\;; 0 \leq k\leq r. \label{eq:Symlimit2}
\end{align}
We note that while $\xi_r \circ \overline{\xi_r}$ is the identity map on $\Symr$, the reverse composition $\overline{\xi_r} \circ \xi_r$ has a nontrivial kernel and is not the identity on all of $\Sym$. In Lemma \ref{lem:commuting} below we resolve this discrepancy by restricting to a particular subalgebra of $\Lambda$. 

It is a beautiful theorem of Schiffmann-Vasserot \cite[Proposition 1.4]{SchiffmannVasserotK} (see also \cite[Section 3.2]{GorskyNegut}) that the inverse limits in  \eqref{eq:eha limit} and \eqref{eq:Symlimit} are compatible and give rise to a action of $\EHA$ on $\Lambda$ that is obtained as the inverse limit of $\pol{r}$. We call this the \emph{polynomial representation} of $\EHA$ and denote it by:
\begin{equation}
 \Pol: \EHA \to \End(\Sym).
\end{equation}

\begin{proposition}[\cite{FeiginTsymbaliuk, GorskyNegut, SchiffmannVasserotK}]\label{prop:iso reps}
The representations $\Pol$ and $\Geom$ of $\EHA$ on $\Sym$ are isomorphic. That is, for all $P_\bx \in \EHA$ we have $\Phi \circ \Pol(P_\bx)  = \Geom(P_\bx)\circ  \Phi $ where $\Phi: \Sym \to \Sym$ is the following plethystic substitution:
\[
 f(X) \mapsto f\left[\frac{X}{1-t^{-1}}; q,t\right].
 \]
\end{proposition}

We remark that by \cite[Proposition 1.4]{SchiffmannVasserotK} the eigenfunctions of the operators $P_{0,\ell} \in \EHA$ are the Macdonald polynomials $P_\lambda(X;q,t)$, which indeed are mapped to $\tilde{H}_{\lambda}(X;q,t)$ under the plethysm in Proposition \ref{prop:iso reps}.

\subsection{A Finite Rational Shuffle Theorem}
In order to construct a finite rational shuffle theorem for $\DAHA{r}$, our first goal is to transport the geometric representation of $\EHA$ to $\Symr$. This follows easily from composition with the inverse limit maps in \eqref{eq:Symlimit}
and \eqref{eq:Symlimit2}.

\begin{proposition} \label{prop:EHA-restricted action}
Consider the representation $\Psi_r^{G}: \EHA \to \End(\Symr)$ given by 
\[
\Psi_r^{G}(P_\bx):= \xi_r \circ \Geom(P_\bx) \circ \overline{\xi}_r
\]
 for each $P_\bx \in \EHA$. Then, for every pair of coprime integers $m,n \in \Z_{\geq0}$, 
\[
\Psi_r^{G}(P_{(m,n)}) \cdot 1 = C^{(r)}_{(m,n)}( x_1, \dots, x_r; q,t).
\]
\end{proposition}

\begin{proof}
The proof is immediate from Theorems \ref{thm:shuffle} and  \ref{thm:Cat=Hik}, since by definition \[\Psi_r^{G}(P_{(m,n)}) \cdot 1 = \xi_r \circ \Geom(P_{(m,n)}) \circ \overline{\xi}_r (1)\] where  $\overline{\xi}_r (1)=1$  and $\xi_r \left(\Hik_{(m,n)}(X;q,t)\right)= \Hik_{(m,n)}(x_1,\dots,x_r ;q,t)= C^{(r)}_{(m,n)}( x_1, \dots, x_r; q,t)$.
\end{proof}

Our second goal is to prove $\Psi_r^{G}$ is isomorphic to the analogous ``truncated polynomial representation" of $\EHA$ on $\Symr$. To do this we must define an analogue of the plethysm $\Phi$ in Proposition \ref{prop:iso reps} for symmetric polynomials. Of course, plethysm is generally defined only at the level of functions, where an infinite number of variables are available. Nonetheless, utilizing the fact that power sum symmetric polynomials $p_1, \dots, p_r$ generate $\Symr$ we consider the following map. 

Define $\overline{\Phi}: \Symr \to \Symr$ by:
\begin{equation}\label{eq:phibar}
\overline{\Phi}: p_k(x_1,\dots,x_r) \mapsto \frac{1}{1-t^{-k}}\; p_k(x_1,\dots,x_r) ; \qquad  0 \leq k \leq r.
\end{equation}
Extending linearly, it is straightforward to see that $\overline{\Phi}$ induces an algebra isomorphism on $\Symr$. 

\begin{lemma}\label{lem:commuting}
Consider the subalgebra $\Sym^{(r)} \subset \Sym$ generated by $p_1(X), \dots , p_r(X)$. Then the following diagrams commute:
\[
\begin{tikzcd}
\Sym^{(r)} \arrow[r, "\Phi"] \arrow[d, "\xi_r"] &\Sym^{(r)}\arrow[d, "\xi_r"] \\
\Symr\arrow[r, "\overline{\Phi}"] & \Symr
\end{tikzcd}
\qquad\qquad
\begin{tikzcd}
\Symr \arrow[r, "\overline{\Phi}"] \arrow[d, "\overline{\xi_r}"] &\Symr\arrow[d, "\overline{\xi_r}"] \\
\Sym\arrow[r, "{\Phi}"] & \Sym
\end{tikzcd}.
\]
That is,  $\xi_r \circ \Phi  = \overline{\Phi} \circ \xi_r$ on $\Sym^{(r)}$ and $\overline{\xi_r} \circ \overline{\Phi} = \Phi \circ \overline{\xi_r}$ on $\Symr$, respectively. In particular, $\overline{\xi_r} \circ \xi_r$ is the identity on $\Sym^{(r)}$. 
\end{lemma}

\begin{proof}
All claims follow from direct verification on the generating sets $\lbrace p_1, \dots, p_r \rbrace$ and the fact that all maps are algebra homomorphisms. 
\end{proof}

The restriction to the subalgebra $\Sym^{(r)}$ is further motivated by the fact that it is an invariant subspace under the polynomial action of the elliptic Hall algebra.

\begin{lemma} \label{lem:preserve}
The polynomial representation $\Pol$ of $\EHA$ preserves the subalgebra $\Lambda^{(r)}$ of $\Lambda$. 
\end{lemma}

\begin{proof}
By the results of \cite[Section 4.4]{Bergeron}, $\EHA$ is generated by $P_{(0,1)}$ and $P_{(1,0)}$, so it suffices to prove that these two generators preserve $\Lambda^{(r)}$.

From \cite[Prop. 1.4]{SchiffmannVasserotK} we know that $P_{(1,0)}$ acts on $\Lambda$ by multiplication by $p_1(X)$ and that $P_{(0,1)}$ acts by Macdonald's difference operator with eigenfunctions given by $P_\lambda(X;q,t)$. It is obvious that $P_{(1,0)}$ will preserve $\Lambda^{(r)}$ for any $r>0$. To see that this subspace is also invariant under $P_{(0,1)}$ we first claim that $\Lambda^{(r)}$ is spanned over $\C(q,t)$ by the Macdonald polynomials $P_\lambda(X; q, t)$ where $\lambda$ has at most $r$ parts. Indeed, according to \cite[Section 2]{SchiffmannVasserotK}, the map $\Lambda_{r+1} \to \Lambda_r$ sending the variable $x_{r+1}$ to $0$ sends $\xi_{r+1}(P_\lambda(X; q, t))$ to $\xi_{r}(P_\lambda(X; q, t))$ if $\lambda$ has at most $r$ parts; and to zero otherwise. It follows that the map $\overline{\xi_{r}}: \Lambda_r \to \Lambda$ sends $\xi_r(P_{\lambda}(X; q, t))$ to $P_{\lambda}(X; q, t)$ if $\lambda$ has at most $r$ parts; and to zero otherwise. By Lemma \ref{lem:commuting},  since $\xi_{r}|_{\Lambda^{(r)}}$ and $\overline{\xi_{r}}$ are inverse isomorphisms between $\Lambda_r$ and $\Lambda^{(r)}$, the claim follows. Now the fact that $P_{(0,1)}$ preserves $\Lambda^{(r)}$ follows from the fact that Macdonald polynomials are eigenvectors of $P_{(0,1)}$. 
\end{proof}

Analogously to $\Psi_r^{G}$ in Proposition \ref{prop:EHA-restricted action}, let $\Psi_r^{P}$ be the representation of $\EHA$ on $\Symr$ defined on each $P_\bx \in \EHA$ by:
\begin{equation}
\Psi_r^{P}(P_\bx):=\xi_r \circ \Pol(P_\bx) \circ \overline{\xi}_r. 
\end{equation}

\begin{proposition}\label{prop:finisoreps}
The representations $\Psi_r^{P}$ and $\Psi_r^{G}$ of $\EHA$ on $\Symr$ are isomorphic. In particular, for all $P_\bx \in \EHA$ we have $ \overline{\Phi}\circ \Psi_r^P(P_x) = \Psi_r^G(P_x) \circ \overline{\Phi}$, where $\overline{\Phi}$ is as in \eqref{eq:phibar}.
\end{proposition}

\begin{proof}
From the definition of $\overline{\xi_r}$, it is easy to see that $\overline{\xi_r}(\Lambda_r) \subseteq \Lambda^{(r)}$.  Therefore, combining Lemmas \ref{lem:commuting} and \ref{lem:preserve} with the fact that $\Phi$ acts by scaling power sums and thus preserves $\Lambda^{(r)}$, we see that on $\Lambda_r$ the following equalities hold:
\begin{align*}
\Psi_r^G(P_x)
&= \xi_r \circ \Geom(P_x) \circ \overline{\xi_r}  \\
&= \xi_r \circ\Phi \circ \Pol(P_x) \circ \Phi^{-1} \circ \overline{\xi_r} \\
&= \overline{\Phi}  \circ \xi_r \circ \Pol(P_x) \circ \overline{\xi_r} \circ  \overline{\Phi}^{-1} \\
&= \overline{\Phi}  \circ \Psi_r^P(P_x) \circ  \overline{\Phi}^{-1}.
\end{align*}
\end{proof}

Finally, we are ready to construct the desired representation. In particular, we will show that the restricted geometric representation $\Psi_r^G$ of $\EHA$ induces a representation of $\DAHA{r}$ on $\Symr$ that is isomorphic to the polynomial representation $\pol{r}$ of $\DAHA{r}$ previously discussed in Section \ref{subsec:SphDaha}. Thus, combining Propositions \ref{prop:EHA-restricted action} and \ref{prop:finisoreps} gives us a finite version of the Rational Shuffle Theorem.

\begin{theorem} \label{thm:FiniteShuffle}
There is an action $\Psi_r^{G}$ of $\DAHA{r}$ on $\Symr$ such that for every pair of coprime integers $(m, n) \in \Z^{2}_{\geq 0} \setminus \lbrace (0,0)\rbrace$ we have:
\begin{equation}\label{eq:finiteshuffle}
P^{(r)}_{(m,n)}\cdot 1 = C^{(r)}_{(m,n)}( x_1, \dots, x_r; q,t).
\end{equation}
This action is isomorphic to, but does not coincide with, the polynomial representation $\pol{r}$. 
\end{theorem}

\begin{proof}
Denote by $\pi_r: \EHA \twoheadrightarrow \DAHA{r}$ the surjective maps that give rise to the inverse limit. Then 
by \cite[Prop. 1.4]{SchiffmannVasserotK} we know that $\Psi_r^{P} = \pol{r} \circ \pi_r$, thus we have equivalent sequences,
\[ \left( \ker(\Psi_r^{P}) \hookrightarrow \EHA \overset{\Psi_r^{P}}{\longrightarrow 
} \End(\Lambda_r) \right) = 
\left( \ker(\Psi_r^{P}) \hookrightarrow \EHA  \overset{\pi_r}{\twoheadrightarrow} \DAHA{r} \overset{\Pol_r}{\longrightarrow} \End(\Lambda_r)\right) .\]
Since $\Pol_r$ is faithful, it follows that $\ker(\Psi_r^{P}) = \ker(\Pol_r \circ \pi_r) = \ker(\pi_r)$, hence
\[\ker(\Psi_r^{P}) \hookrightarrow \EHA  \overset{\pi_r}{\twoheadrightarrow} \DAHA{r} \]
 is exact. Thus, $\EHA / \ker(\Psi_r^{P}) \cong \DAHA{r}$. By Proposition \ref{prop:finisoreps}, $\Psi_r^P$ and $\Psi_r^G$ are isomorphic, hence have equal kernels and so $\EHA / \ker(\Psi_r^G) \cong \DAHA{r}$. 
 
 Thus, under the isomorphism $\EHA / \ker(\Psi_r^{G}) \cong \DAHA{r}$ the $\EHA$ representation $\Psi_r^{G}$ induces an action of $\DAHA{r}$ on $\Symr$ where $P_{(m,n)}^{(r)}$ acts by $\Psi_r^G([P_{(m,n)}])$. Here $[P_{(m,n)}]$ denotes the class of $P_{(m,n)}$ in $\EHA / \ker(\Psi_r^G)$. Hence, the action of $P^{(r)}_{(m,n)}$ is given by $\Psi_r^G(P_{(m,n)})$, so by Proposition \ref{prop:EHA-restricted action} equation \eqref{eq:finiteshuffle} follows. 

Now, by construction, the $\DAHA{r}$ representation that $\Psi_r^{P}$ induces is precisely the polynomial representation $\pol{r}$. Consequently, since the isomorphism in Proposition \ref{prop:finisoreps} passes through the quotient, then the action of $\DAHA{r}$ induced by $\Psi_r^{G}$ is clearly isomorphic to its classical polynomial representation $\pol{r}$ via the same isomorphism $\overline{\Phi}$. 
\end{proof}

Let us remark that computing explicitly the action of $\DAHA{r}$ on $\C(q,t)[x_1, \dots, x_r]^{S_{r}}$ given by Theorem \ref{thm:FiniteShuffle} is a daunting task. For example, when $r = 1$, the spherical DAHA $\DAHA{1}$ is the algebra of difference operators on the line, that is, $\DAHA{1}$ is generated by elements $X, Y$ with a single relation $YX = qXY$. Moreover, the elements $P_{(m,n)}^{(1)} \in \DAHA{1}$ have the form:
\[
P_{(m,n)}^{(1)} = q^{\alpha(m,n)}Y^{n}X^{m}
\] 
for some integer $\alpha(m,n)$. 

According to Theorem \ref{thm:FiniteShuffle}, the algebra $\DAHA{1}$ acts on $\C(q,t)[x]$ in such a way that, for each pair of coprime positive integers $(m,n)$:
\[
P_{(m,n)}\cdot 1 = C_{(m,n)}(q,t)x^{m}
\]
where $C_{(m,n)}(q,t)$ is the usual $(q,t)$-Catalan number. The algebra $\DAHA{1}$ and the elements $P_{(m,n)}^{(1)}$ are quite simple while the $(q,t)$-Catalan numbers are quite complicated. Thus, the action of $\DAHA{1}$ on $\C(q,t)[x]$ is bound to be complicated. 

\section{The non-coprime case}\label{sec:noncoprime}

The purpose of this section is to find a combinatorial formula for the rank $r$ $(m,n)$-rational Catalan numbers $C^{(r)}_{(m,n)}$ in the case when $m$ and $n$ are \emph{not coprime}. This is in the spirit of, and heavily inspired by, Bizley's formula for $C_{(m,n)}$ in the noncoprime case \cite{bizley}. 

\subsection{Maxima} Let us fix $m, n, r > 0$ (with $m$, $n$ not necessarily coprime) and recall from the proof of Theorem \ref{thm:coprime sspf} the set
\begin{equation}
\Z_{\geq 0}^{n \times r}(m) = \left\{(a_{11}, \dots, a_{1r}| \dots| a_{n1}, \dots, a_{nr}) \in \Z^{nr}_{\geq 0} \mid \sum_{a_{ij}} = m\right\}.
\end{equation}
Recall also that $a_{i} := \sum_{j = 1}^{r}a_{ij}$. Let us say that the element $k \in \{1, \dots, n\}$ is a \emph{maximum} of $(a_{ij})$ if it maximizes the quantity
\begin{equation}
t_{k}(a_{ij}) := \frac{mk}{n} - a_{1} - \cdots - a_{k}.
\end{equation}
Note that $t_{k}(a_{ij}) + m$ is the $y$-intercept of the line with slope $-m/n$ and passing through the point $(k, m-a_{1} - \cdots - a_{k})$. It follows that, upon the identification $\Z_{\geq 0}^{n \times r}(m)$ with $P^{-}_{r}(m,n)$ defined in the proof of Theorem \ref{thm:coprime sspf}, the tuple $(a_{ij})$ defines a semistandard parking function if and only if $n$ is a maximum of $(a_{ij})$, in other words, if and only if $t_{k}(a_{ij}) \leq 0$ for every $k$. Let us now study the behavior of maxima under the action of $\Z/(n) = \langle \rho \mid \rho^n = 1\rangle$.

\begin{lemma}\label{lem:maxima}
Assume $k$ is a maximum of $(a_{ij}) \in \Z_{\geq 0}^{n \times r}(m)$. Then $k-1$ (resp. $k+1$) is a maximum of $\rho\cdot(a_{ij})$ (resp. $\rho^{-1}\cdot(a_{ij})$. In other words, $k$ is a maximum of $(a_{ij})$ if and only if $k-1$ is a maximum of $\rho\cdot (a_{ij})$. 
\end{lemma}
\begin{proof}
We only prove the statement concerning $\rho$, the one concerning $\rho^{-1}$ is analogous. Let us denote $t := t(a_{ij})$ and $t' := t(\rho\cdot (a_{ij}))$. We know that $t_{k} \geq t_{p}$ for any $p = 1, \dots, n$ but must additionally show that $t'_{k-1} \geq t'_{p}$ for all such values of $p$.
Let us assume, first, that $k \neq 1$. Then,
\[
t'_{k-1} = \frac{m(k-1)}{n} - a_{2} - \cdots - a_{k} = t_{k} -\frac{m}{n} + a_{1} \geq t_{p} - \frac{m}{n} + a_{1} = t'_{p-1}
\]
for $p \neq 1$. Thus, it only remains to justify that $t'_{k-1} \geq t'_{n} = 0$. We know that $t_{k} \geq t_{1} = \frac{m}{n} - a_{1}$, so $t'_{k-1} = t_{k} - \frac{m}{n} + a_{1} \geq 0 = t'_{n}$. The result follows. 

Now assume that $k = 1$, so we need $t'_{n} \geq t'_{p}$ for any $p = 1, \dots, n$. That is, $0 \geq t'_{p}$ for $p = 1, \dots, n-1$. This is equivalent to the statement that $0 \geq t_{p} - \frac{m}{n} + a_{1}$ for $p = 2, \dots, n$, which in turn is the same as $t_{1} \geq t_{p}$ for $p = 2, \dots, n$. This finishes the proof.
\end{proof}

\begin{corollary}\label{cor:maxima}
The action of $\Z/(n)$ on $\Z^{n \times r}_{\geq 0}(m)$ preserves the number of maxima. 
\end{corollary}

\subsection{A Bizley-type formula} Now we are ready to derive a Bizley-type formula for the higher rank Catalan number in the non-coprime case. Our formula and proof follow closely those of Bizley's \cite{bizley}. Let us start with coprime $m, n$ and fix $r > 0$. Consider a rectangle of size $mk \times nk$. For $t = 1, \dots, k$, let $\varphi_{k, t}$ be the number of semistandard parking functions whose underlying Dyck path touches the diagonal in $t$ points. It is clear that
\begin{equation}\label{eqn:higher rank as sum}
C^{(r)}_{(km, kn)} = \sum_{t = 1}^{k} \varphi_{k,t}.
\end{equation}
Similarly to the proof of the main result of \cite{bizley}, it follows from Corollary \ref{cor:maxima} that the number of elements $(a_{ij}) \in \Z^{n \times r}_{\geq 0}(m)$ with exactly $t$ maxima is
\[
\frac{1}{t}nk\varphi_{k,t}.
\]
It also follows from Corollary \ref{cor:maxima} that the total number of maxima an element $(a_{ij}) \in \Z^{nk \times r}_{\geq 0}(mk)$ may have ranges from $1, \dots, k$. Thus,
\[
|\Z^{nk \times r}_{\geq 0}(mk)| = \binom{nkr + mk - 1}{mk} = \sum_{t=1}^{k}\frac{1}{t}nk\varphi_{k,t} \Rightarrow \frac{1}{nk}\binom{nkr+mk-1}{mk} = \sum_{t = 1}^{k}\frac{1}{t}\varphi_{k,t}.
\]
Now let $\psi_{\ell} := \varphi_{\ell, 1}$. That is, $\psi_{\ell}$ is the number of rank $r$ semistandard parking functions on an $m\ell \times n\ell$-rectangle whose underlying Dyck path only touches the diagonal at one point, necessarily located at the southeast corner of the rectangle. Hence, we have 
\begin{equation}\label{eqn:varphi as sum}
\varphi_{k, t} = \sum_{\substack{\ell_{i} \geq 0 \\ \ell_{1} + \cdots + \ell_{t} = k}}\psi_{\ell_{1}}\cdots \psi_{\ell_{t}}.
\end{equation}
This is precisely is the coefficient of $x^{k}$ in $P(x)^{t}$, where
\[
P(x) = \sum_{\ell = 1}^{\infty}\psi_{\ell}x^{\ell}.
\]
So then $\frac{1}{nk}\binom{nkr+mk-1}{mk}$ is the coefficient of $x^{k}$ in $\sum_{t = 1}^{k}\frac{1}{t}P(x)^{t}$. Since $P(x)$ has no constant term, then this is also the coefficient of $x^{k}$ in $\sum_{t = 1}^{\infty}\frac{1}{t}P(x)^{t}$. Thus, we have
\[
\sum_{k = 1}^{\infty}\frac{1}{nk}\binom{nkr+mk-1}{mk}x^{k} = \sum_{t=1}^{\infty}\frac{1}{t}P(x)^{t} = -\log(1 - P(x)).
\]
Equivalently,
\begin{equation}\label{eqn:p is exp}
P(x) = 1 - \exp\left(-\sum_{k = 1}^{\infty}\frac{1}{nk}\binom{nkr+mk-1}{mk}x^{k} \right).
\end{equation}
Now it follows from \eqref{eqn:higher rank as sum} and \eqref{eqn:varphi as sum}, as well as from $P(x)$ having no constant term, that
\begin{equation}\label{eqn:geom series}
\sum_{k = 1}^{\infty}C^{(r)}_{(mk, nk)}x^{k} = \sum_{t = 1}^{\infty}P(x)^{t} = \frac{1}{1 - P(x)} -1.
\end{equation}
Thus, from \eqref{eqn:p is exp} and \eqref{eqn:geom series} we arrive at the following. 

\begin{theorem}[Bizley's formula]
Let $m, n$ be coprime positive integers, and $r\in \Z_{>0}$. Then we have an equality of formal power series:
\[
1 + \sum_{k = 1}^{\infty}C^{(r)}_{(mk,nk)}x^{k} = \exp\left(\sum_{k = 1}^{\infty}\frac{1}{nk}\binom{nkr+mk-1}{mk}x^{k} \right).
\]
\end{theorem}

We note that one can apply similar techniques to obtain a Bizley-type formula for the number of parking functions. For coprime $m$ and $n$,
\[
1 + \sum_{k = 1}^{\infty}\frac{1}{(mk)!}|\park{mk,nk}| = \exp\left(\sum_{k = 1}^{\infty}\frac{(nk)^{mk-1}}{(mk)!}x^{k}\right).
\]
This is, however, much more elegantly expressed by Aval and Bergeron in \cite{AvalBergeron}, who adapt Bizley's proof incorporating symmetric function arguments to find the Frobenius characteristic of the $S_{mk}$-representation $\C\!\park{mk,nk}$. One can adapt the arguments in that paper to show that
\[
C^{(r)}_{(mk,nk)} = \operatorname{Frob}(\C\!\park{mk, nk})|_{x_{1} = \cdots = x_{mk} = 1, x_{mk+1} = \cdots = 0}.
\]

\appendix

\section{On the computation of the dinv statistic}\label{sec:appendix}

In this Appendix, we elaborate on the computation of $\dinv(\D, \Upsilon)$, where $(\D, \Upsilon)$ is a semistandard parking function. In particular, we will give a ``laser ray" formula for $\codinv$, and we will show that, in the case when $\Upsilon$ is constant, $\dinv(\D, \Upsilon) = \dinv(\D)$, where the dinv statistic on Dyck paths is as in e.g. \cite{Loehr}. For the rest of this appendix, we fix coprime numbers $m$ and $n$, all the Dyck paths we consider will be in $\dyck{m,n}$.

\subsection{Codinv in terms of inversions} Let $(\D, \Upsilon)$ be a semistandard parking function of weight $\bw$, and $f = \mathcal{A}_{\bw}(\D, \Upsilon)$. Recall that $\dinv$ (and thus also $\codinv$) of $(\D, \Upsilon)$ is defined in terms of its standardization, which is the parking function $(\D, \varphi)$ associated to the unique minimal-length right $S_{\bw}$-coset representative in $\affsym{m}{}$ $\sigma$ satisfying $f = f_{\bw}\sigma$. In turn, $\codinv(\D, \varphi)$ is defined in terms of the inversions of $\sigma$, see e.g. \cite{GMV}. But by Lemma \ref{lem:inversions}, the inversions of $f$ and of $\sigma$ coincide. Thus, we get the following alternative definition of $\codinv$. 

\begin{lemma}\label{lem:dinv inversions}
Let $(\D, \Upsilon)$ be a semistandard parking function of weight $\bw$, and $f := \mathcal{A}_{\bw}(\D, \Upsilon) \in \affc{\bw}^{n}(m, r)$. Then,
\[
\codinv(\D, \Upsilon) = |\{(i,j) \in \{1, \dots, m\} \times \{1, \dots, n\} \mid f(i+j) < f(i)\}|.
\]
\end{lemma}

\begin{remark}\label{rmk:dinv inversions}
We will use Lemma \ref{lem:dinv inversions} in a slightly different way. First, recall that the definition of $f = \mathcal{A}_{\bw}(\D, \Upsilon)$ is in two steps, first defining a function $\tilde{f}$ and $f$ is defined as a translation of $\tilde{f}$. Thus, we could have replaced $f$ by $\tilde{f}$ in Lemma \ref{lem:dinv inversions}.

Second, note that if $\tilde{f}(i+j) < \tilde{f}(i)$, then $\tilde{f}(i+m+j) = \tilde{f}(i+j) + r < \tilde{f}(i) + r = \tilde{f}(i+m)$. Thus, in Lemma \ref{lem:dinv inversions} we can replace the set $\{1, \dots, m\}$ by any $m$-element set whose elements are pairwise distinct modulo $m$. In other words, if $\{a_{1}, \dots, a_{m}\}$ is any set of elements which are pairwise distinct mod $m$ then
\[
\codinv(\D, \Upsilon) = |\{(i,j) \in \{1, \dots, m\} \times \{1, \dots, n\} \mid \tilde{f}(a_i+j) < \tilde{f}(a_i)\}|.
\]
This is the form in which we will use Lemma \ref{lem:dinv inversions}.
\end{remark}

\subsection{Path and area boxes}  Recall that a box is a unit square whose corners have integer coordinates, and that we identify a box with its upper right corner.

\begin{definition}
We say that a box is an \emph{area box} of $\D$ if it is entirely contained between $\D$ and the main diagonal. We say that a box is a \emph{path box} of $\D$ if its right side is a vertical step of $\D$. 
\end{definition} 

Note that $\D$ always has exactly $m$ path boxes, corresponding to its vertical steps, see Figure \ref{fig:boxes}. 

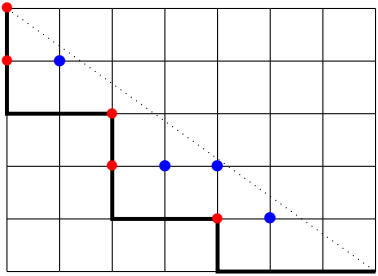
\begin{figure}[ht]
\begin{tikzpicture}[scale=0.7]
 \pgfmathtruncatemacro{\m}{5}
  \pgfmathtruncatemacro{\n}{7}
 \pgfmathtruncatemacro{\mi}{\m-1} 
    \pgfmathtruncatemacro{\ni}{\n-1} 
    %
\draw[dotted] (0,\m)--(\n,0); 
  \foreach \x in {0,...,\n} 
    \foreach \y  [count=\yi] in {0,...,\mi}    
      \draw (\x,\y)--(\x,\yi);
        \foreach \y in {0,...,\m}     
       \foreach \x  [count=\xi] in {0,...,\ni}          
         \draw (\x,\y)--(\xi,\y);
\draw[line width = 1.5] (0,5)--(0,3)--(2,3)--(2,1)--(4,1)--(4,0)--(7,0);
\node at (0,5) {\small \color{red} $\bullet$};
\node at (0,4) {\small \color{red} $\bullet$};
\node at (2,3) {\small \color{red} $\bullet$};
\node at (2,2) {\small \color{red} $\bullet$};
\node at (4,1) {\small \color{red} $\bullet$};
\node at (1,4) {\color{blue} $\bullet$};
\node at (3,2) {\color{blue} $\bullet$};
\node at (4,2) {\color{blue} $\bullet$};
\node at (5,1) {\color{blue} $\bullet$};
 \end{tikzpicture}
\caption{The (upper right corners of the) area boxes of $\D$, marked as {\color{blue}blue} dots, and the path boxes of $\D$, marked as {\color{red}red} dots.} \label{fig:boxes}
\end{figure}

\subsection{Codinv and laser rays} Let $\D$ be a Dyck path. For any box $b$, its \newword{laser ray} $\ell_{b}$ is the line parallel to the main diagonal and passing through the upper right corner of $b$.

Now let $(\D, \Upsilon)$ be a semistandard parking function. If $\color{blue} b$ is an area box of $\D$, let
\begin{equation}
\codinv_{\color{blue} b}(\D, \Upsilon) := |\{v \in \ver(\D) \mid \ell_{{\color{blue}b}} \; \text{intersects} \; v\}|.
\end{equation}

On the other hand, if $\color{red} b$ is a path box of $\D$, let $\Upsilon({\color{red} b})$ be the value of $\Upsilon$ at the vertical step that is the right edge of $\color{red} b$. For such a path box $\color{red} b$ of $\D$, define
\begin{equation}
\codinv_{\color{red} b}(\D, \Upsilon) := |\{v \in \ver(\D) \mid \ell_{{\color{red}b}} \; \text{intersects} \; v \; \text{and} \; \Upsilon({\color{red} b}) < \Upsilon(v)\}|.
\end{equation}

We have then the following computation of $\codinv$.

\begin{proposition}
Let $(\D, \Upsilon)$ be a semistandard parking function. Then,
\[
\codinv(\D, \Upsilon) = \sum_{\substack{{\color{blue} b} \; \text{is an area} \\ \text{box of} \; \D}}\codinv_{{\color{blue} b}}(\D, \Upsilon) +  \sum_{\substack{{\color{red} b} \; \text{is a path} \\ \text{box of} \; \D}}\codinv_{{\color{red} b}}(\D, \Upsilon).
\]
\end{proposition}
\begin{remark}
Note that, if $\color{blue} b$ is an area box of $\D$, then $\codinv_{{\color{blue} b}}(\D, \Upsilon)$ does not depend on $\Upsilon$. 
\end{remark}
\begin{proof}
We will use Lemma \ref{lem:dinv inversions} in the form given in Remark \ref{rmk:dinv inversions}. For the set $\{a_{1}, \dots, a_{m}\}$ we will use $\gamma(v)$, where $v$ is a vertical step of $\D$ (note that this is the same as $\gamma({\color{red} b})$ where $\color{red} b$ is a path box of $\D$). Recall also the horizontal strip $\strip_{m}$ and the extension of $\Upsilon$ to $\strip_{m}$. As noted in Section \ref{sec:affcomp}, the map $\gamma: \strip_{m} \to \Z$ is a bijection, and $\tilde{f}$ is determined by $\tilde{f}(\gamma(b)) = \Upsilon(b)$ for every $b \in \strip_{m}$. 

Let us now pick a path box $\color{red} b$ of $\D$, with coordinates say $(x,y)$. We want to find those boxes $b'$ in $\strip_{m}$ satisfying:
\begin{enumerate}
\item $\gamma({\color{red} b}) < \gamma(b') < \gamma({\color{red} b}) + n$,
 \item $\Upsilon(b') < \Upsilon({\color{red} b})$.
\end{enumerate}
Note, first, that $b'$ cannot be on the same vertical nor horizontal line as $\color{red} b$. We assume, first, that $b'$ is southeast of $\color{red} b$. This means that $b' = (x+x', y-y')$ for some $x', y' \in \Z_{> 0}$. Now $\gamma(b') = \gamma({\color{red} b}) -mx' + y'n$, so the condition $\gamma({\color{red} b}) < \gamma(b') < \gamma({\color{red} b}) + n$ translates to
\[
\frac{-y'}{x'} < \frac{-m}{n} < \frac{-y'+1}{x'}.
\]
which means that the laser ray of $\color{red} b$ cuts the length-1 vertical segment whose bottom vertex is $b'$. Equivalently, the laser ray starting at $b'$ cuts the vertical step whose top vertex is ${\color{red} b'}$. It is easy to see that we get the same condition if $b'$ is northwest of $b$, and the condition $\gamma({\color{red} b}) < \gamma(b')$ precludes $b'$ from being northeast or southwest of $\color{red} b$. Thus, Condition (1) translates to the laser ray of $b'$ intersecting the vertical step whose top vertex is $\color{red} b$.

Let us move to Condition (2). Note that $\Upsilon$ takes values in $[1,r]$ on vertical steps of $\D$, values $> r$ on lattice points to the left of $\D$, and values $< 0$ on  lattice points to the right of $\D$. Thus, $b'$ must be either on or to the right of $\D$. It is clear that the laser ray of a point to the right of the main diagonal will not intersect any vertical step of $\D$. Thus, $b'$ is either an area box or a path box of $\D$. If $b'$ is an area box of $\D$ the condition $\Upsilon(b') < \Upsilon({\color{red} b})$ is automatically satisfied. If $b'$ is a path box of $\D$, we need to ask for this condition. From here, the result follows. 
\end{proof}

In Figure \ref{fig:codinv laser} we compute the $\codinv$ of a semistandard parking function using laser rays.

\begin{figure}[ht]
\begin{tikzpicture}[scale=1.5]
\draw(0,0)--(7,0)--(7,5)--(0,5)--cycle;
\draw(0,1)--(7,1);
\draw(0,2)--(7,2);
\draw(0,3)--(7,3);
\draw(0,4)--(7,4);
\draw(1,0)--(1,5);
\draw(2,0)--(2,5);
\draw(3,0)--(3,5);
\draw(4,0)--(4,5);
\draw(5,0)--(5,5);
\draw(6,0)--(6,5);
\draw[dotted](7,0)--(0,5);
\draw[line width = 1.5](0,5)--(0,3)--(2,3)--(2,1)--(4,1)--(4,0)--(7,0);

\draw[dashed, color=red] (0,4)--(28/5,0);
\node at (0,4) {\small{\color{red} $\bullet$}};
\draw[dashed,thick,  color=red] (0, 3+10/7)--(31/5,0);
\node at (2,3) {\small{\color{red} $\bullet$}};
\draw[dashed,thick,  color=red] (0, 2+10/7)--(24/5,0);
\node at (2,2) {\small{\color{red} $\bullet$}};
\draw[dashed, color=red] (0,1+20/7)--(27/5,0);
\node at (4,1) {\small{\color{red} $\bullet$}};

\draw[dashed,thick,  color=cyan] (0,4+5/7)--(33/5,0);
\node at (1,4) {\small{\color{blue} $\bullet$}};
\draw[dashed,thick,  color=cyan] (0, 2+15/7)--(29/5,0);
\node at (3,2) {\small{\color{blue} $\bullet$}};
\draw[dashed,thick,  color=cyan] (0, 2+20/7)--(34/5,0);
\node at (4,2) {\small{\color{blue} $\bullet$}};
\draw[dashed, thick, color=cyan] (0, 1+25/7)--(32/5,0);
\node at (5,1) {\small{\color{blue} $\bullet$}};

\draw[red, line width = 3pt ]  (-0.1, 3+10/7)-- (0.1, 3+10/7);
\draw[red, line width = 3pt ]  (-0.1, 2+10/7)-- (0.1, 2+10/7);
\draw[red, line width = 3pt ]  (4-0.1, 2-10/7)-- (4+0.1, 2-10/7);
\draw[cyan, line width = 3pt ]  (-0.1, 2+20/7)-- (0.1, 2+20/7);
\draw[cyan, line width = 3pt ]  (-0.1, 4+5/7)-- (0.1, 4+5/7);
\draw[cyan, line width = 3pt ]  (-0.1, 2+15/7)-- (0.1, 2+15/7);
\draw[cyan, line width = 3pt ]  (2-0.1, 2+5/7)-- (2+0.1, 2+5/7);
\draw[cyan, line width = 3pt ]  (-0.1, 1+25/7)-- (0.1, 1+25/7);

\node at (0.2, 4.5) {\Large \color{blue} 2};
\node at (0.2, 3.5) {\Large \color{blue} 2};
\node at (2.2, 2.5) {\Large \color{blue} 1};
\node at (2.2, 1.5) {\Large \color{blue} 1};
\node at (4.2, 0.5) {\Large \color{blue} 3};
\end{tikzpicture}
\caption{Computing the $\codinv$ statistic of a semistandard parking
function using laser rays.  Every intersection of a
{\color{cyan}{blue}}
 laser ray with a vertical step adds 1 to $\codinv$; there are five of
these. An intersection of a
{\color{red}{red}}
laser ray with a vertical step adds 1 to $\codinv$ only if the laser ray emanates from a path box whose value under $\Upsilon$ is smaller than that of the vertical step the ray intercepts. There are three contributions to $\codinv$ appearing this way. In total, $\codinv(\D, \Upsilon) = {\color{cyan} 5} + {\color{red} 3 } = 8$.}
\label{fig:codinv laser} \end{figure}
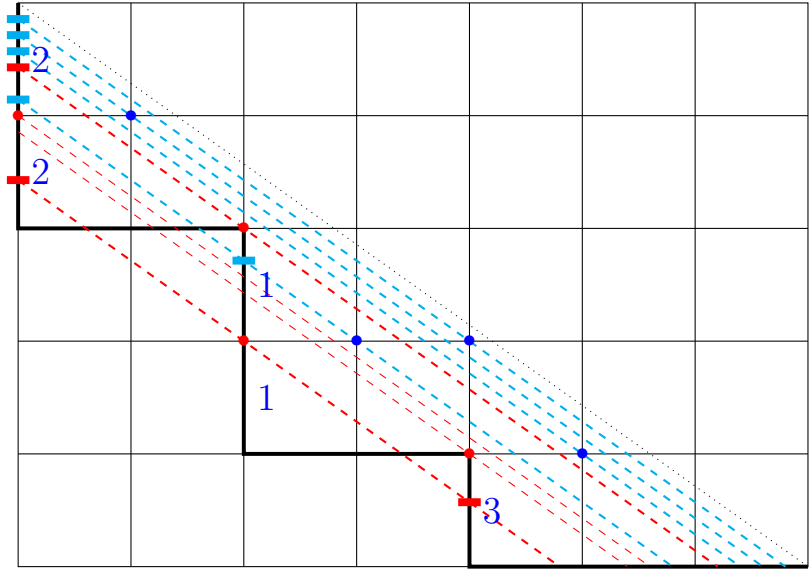

\bibliographystyle{acm} 
\bibliography{references}

\end{document}

%% file: packages-Catalan.tex
\usepackage{extsizes}

\usepackage{graphicx}
\usepackage{amsfonts,amssymb,wasysym}
\usepackage{color}
\usepackage{textcomp}
\usepackage{mathtools}
\usepackage{ytableau}
\usepackage{stmaryrd}
\usepackage{soul}
\usepackage{tikz-cd}
\usepackage[utf8]{inputenc}
\usepackage{cmap}
\usepackage{tikz,environ}
\usepackage{float}
\usetikzlibrary{cd}
\usepackage{dsfont}
\usetikzlibrary{arrows}
\usepackage{multicol}
\usepackage{wrapfig}
\usepackage{tcolorbox}
\usepackage{ifthen}

\usepackage{mathdots}

\usepackage{expl3}
\expandafter\def\csname ver@l3regex.sty\endcsname{}

\usepackage{etoolbox}

\usetikzlibrary{knots,snakes}
\usetikzlibrary{hobby}

\usepackage{url}

\usepackage[bookmarks=true,
    colorlinks=true,
    linkcolor=blue,
    citecolor=blue,
    filecolor=blue,
    menucolor=blue,
    urlcolor=blue,
    breaklinks=true]{hyperref}

\usepackage{fullpage}

\usepackage{verbatim}

%% file: macros-Catalan.tex
\input xy
\usepackage[all]{xy}
\xyoption{line}
\xyoption{arrow}
\xyoption{color}
\SelectTips{cm}{}

\usepackage{tikz}
\usetikzlibrary{decorations.markings}
\usetikzlibrary{decorations.pathreplacing}

\tikzstyle directed=[postaction={decorate,decoration={markings,
    mark=at position #1 with {\arrow{>}}}}]
\tikzstyle rdirected=[postaction={decorate,decoration={markings,
    mark=at position #1 with {\arrow{<}}}}]

\usetikzlibrary{calc}
\usepackage{relsize}

\tikzset{fontscale/.style = {font=\relsize{#1}}
    }

\usetikzlibrary{decorations.pathmorphing}
\usetikzlibrary{decorations.text}

\tikzset{snake it/.style={decorate, decoration=snake}}

\usepackage{graphicx}
\usepackage{color}

\usepackage{bbm}
\def\C{{\mathbb{C}}}
\def\R{{\mathds R}}

\def\K{{\mathcal{K}}}
\def\cA{\mathcal{A}}

\def\1{{\mathbbm{1}}}
\def\D{\mathsf{D}}

\newcommand{\Z}{\mathbb{Z}}

\newcommand{\Hik}{\mathcal{H}} 

\DeclareMathAlphabet{\mathpzc}{OT1}{pzc}{m}{it}

\newcommand{\action}{\mathbf{a}}

\usepackage[normalem]{ulem}
\usepackage[colorinlistoftodos]{todonotes}
\DeclareMathOperator{\sgn}{sgn}
\DeclareMathOperator{\Inv}{Inv}
\DeclareMathOperator{\std}{\mathsf{std}}
\newcommand{\sstd}[1]{\operatorname{sstd}^{#1}}

\newcommand{\wt}{\ensuremath\mathrm{wt}}

\newcommand{\ep}{\epsilon}
\newcommand{\CO}{\mathcal{O}}

\newcommand{\Fl}{\mathcal{F}\ell}

\DeclareMathOperator{\GL}{GL}
\DeclareMathOperator{\SL}{SL}
\DeclareMathOperator{\spann}{span}
\DeclareMathOperator{\gr}{gr}

\newcommand{\GLr}{\GL_r}

\newcommand{\bw}{\mathbf{w}}
\newcommand\End{\operatorname{End}}

\newcommand{\dyck}[1]{\operatorname{D}(#1)}
\newcommand{\area}{\operatorname{\mathtt{area}}}
\newcommand{\coarea}{\operatorname{\mathtt{co-area}}}
\newcommand{\dinv}{\operatorname{\mathtt{dinv}}}
\newcommand{\codinv}{\operatorname{\mathtt{co-dinv}}}
\newcommand{\qbinom}[2]{\left[\begin{matrix} #1 \\ #2 \end{matrix}\right]_{q}}
\newcommand{\ver}{\mathcal{V}}
\newcommand{\park}[1]{\operatorname{PF}(#1)}
\newcommand{\sspf}[2]{\operatorname{SSPF}_{#1}(#2)}
\newcommand{\we}{\operatorname{\mathtt{wt}}}
\newcommand{\A}{\mathcal{A}}
\newcommand{\affsym}[2]{\widetilde{S}^{#2}_{#1}}

\newcommand{\des}{\operatorname{Des}} 
\newcommand{\hik}{\mathcal{H}}
\newcommand{\affc}[1]{\operatorname{AC}_{#1}}
\newcommand{\affcz}[1]{\overline{\operatorname{AC}}_{#1}}
\newcommand{\strip}{\operatorname{St}}
\newcommand{\minlength}[3]{S_{#1}\backslash\widetilde{S}^{#3}_{#2}}
\newcommand{\EHA}{\mathcal{E}^{++}}
\newcommand{\DAHA}[1]{\mathbb{SH}(#1)^{++}}

\newcommand{\pol}[1]{\mathrm{Pol}_{#1}}

\newcommand{\Sym}{\Lambda}
\newcommand{\Symr}{\Lambda_r}
\newcommand{\bx}{\mathbf{x}}

\DeclareMathOperator{\INV}{Inv}

\newcommand{\AS}{\widetilde{S}}

\newcommand{\gam}{\delta}

\DeclareMathOperator{\Pol}{Pol}
\DeclareMathOperator{\Geom}{Geo}

\newcommand{\newword}[1]{\emph{\textbf{#1}}}

\newcommand{\redone}{{\color{red} 1}}

\newcommand{\redtwo}{{\color{red} 2}}

\newcommand{\blueone}{{\color{blue} 1}}